\documentclass[12pt,a4paper]{article}
\usepackage[utf8]{inputenc}
\usepackage{amsmath}
\usepackage{amsthm}
\usepackage{amsfonts}
\usepackage{amssymb}
\usepackage{color}
\usepackage{natbib}
\usepackage{placeins}
\usepackage{hyperref}
\usepackage{authblk}

\usepackage[left=1.2cm, right=1.2cm]{geometry}

\newcommand{\range}{\text{range}}

\newcommand{\diag}{\text{diag}}
\newcommand{\convunif}{\overset{u}{\to}}

\newcommand{\convlp}{\overset{L^{1}(\mu)}{\to}}

\newcommand{\convas}{\overset{\text{a.e. }\mu}{\to}}
\newtheorem{theorem}{Theorem}
\newtheorem{corollary}{Corollary}
\newtheorem{condition}{Condition}
\newtheorem{definition}{Definition}
\newtheorem{example}{Example}
\newtheorem{lemma}{Lemma}
\newtheorem{proposition}{Proposition}

\author[1]{Ezequiel Smucler}
\author[2]{James M. Robins}
\author[3]{Andrea Rotnitzky}
\affil[1]{Glovo}
\affil[2]{Harvard T.H. Chan School of Public Health}
\affil[3]{University of Washington, Department of Biostatistics}

\begin{document}

\title{On the asymptotic validity of confidence sets for linear functionals of solutions to integral equations}

% \tableofcontents
\maketitle

\begin{abstract}
    This paper examines the construction of confidence sets for parameters defined as linear functionals of a function of $W$ and $X$ whose conditional mean given $Z$ and $X$ equals the conditional mean of another variable $Y$ given $Z$ and $X$.
    Many estimands of interest in causal inference can be expressed in this form, including the average treatment effect in proximal causal inference and treatment effect contrasts in instrumental variable models.
    We derive a necessary condition for a confidence set to be uniformly valid over a model that allows for the dependence between $W$ and $Z$ given $X$ to be arbitrarily weak. Specifically, we show that for any such confidence set, there must exist some laws in the model under which, with high probability, the confidence set has a diameter greater than or equal to the diameter of the parameter’s range. In particular, consistent with the weak instruments literature, Wald confidence intervals are not uniformly valid over the aforementioned model when the parameter's range is infinite. Furthermore, we argue that inverting the score test, a successful approach in that literature, generally fails for the broader class of parameters considered here. We present a method for constructing uniformly valid confidence sets in the special case where all variables, but possibly $Y$, are binary and discuss its limitations. Finally, we emphasize that developing uniformly valid confidence sets for the class of parameters considered in this paper remains an open problem.
\end{abstract}
\section{Introduction}

This work is motivated by the growing interest in causal inference methods that account for unmeasured confounders. Two prominent approaches, instrumental variables and proximal causal inference, identify causal parameters that can be expressed as continuous linear functionals evaluated at a function $g(W,X)$ whose conditional mean given $Z$ and $X$ equals the conditional mean of another variable $Y$ given $Z$ and $X$. 
Substantial progress has been made in estimating these parameters, but the construction of confidence sets that are uniformly valid over so called weak dependence models, that is, models that allow for arbitrarily weak dependence between $W$ and $Z$ given $X$ remains unexplored. This paper aims to fill that gap. 
Note that if $W$ is independent of $Z$ given $X$, then, except in the trivial case where the conditional mean of $Y$ given $Z$ and $X$ does not depend on $Z$, no such function $g$ exists, and the parameter of interest is not well defined. Thus, inference is expected to be challenging under weak dependence models.

Formally, we define the parameter of interest under the assumption that we observe a sample of $n$ independent and identically distributed copies of $O = (Y, Z, W, X)$ from a distribution $P$, where $Y$ is a scalar random variable, and $Z$, $W$, and $X$ are random vectors of dimensions $d_z$, $d_w$, and $d_x$, respectively.
The parameter of interest is defined as
\begin{equation}
    \varphi(P):= E_{P}\left\lbrace m(O,g_{P})\right\rbrace, 
    \label{eq:def_phi}
\end{equation}
where the map $E_{P}\left\lbrace m(O,\cdot)\right\rbrace :L^{2}\lbrace P_{(W,X)}\rbrace \to \mathbb{R}$ is linear and continuous and 
$g_{P}\in L^{2}\lbrace P_{(W,X)}\rbrace$ solves the integral equation 
\begin{equation}
    E_{P}\left\lbrace  g(W,X) \mid Z,X \right\rbrace = E_{P}\left\lbrace Y \mid Z,X \right\rbrace.
    \label{eq:integral_eq}
\end{equation}
Throughout if $V$ is a subvector of $O$, we let $P_{V}$ be the law of $V$ under $P$, and for any measure $\nu$ and $1\leq p <\infty$,  $L^{p}(\nu)$ stands for the usual vector space of functions $f:\mathbb{R} \times \mathbb{R}^{d_{z}} \times \mathbb{R}^{d_{w}}\times \mathbb{R}^{d_{x}}\to \mathbb{R}$ that satisfy $\int \vert f \vert^{p} d\nu < \infty$. 

Estimands of the form \eqref{eq:def_phi} fall within the class studied by \cite{chernozhukov2023simple}. As indicated by these authors and reviewed in Section \ref{sec:model},
examples of such estimands include the average treatment effect in proximal
causal inference \citep{Miao2018,tchetgen2020,kallus} and linear functionals of the nonparametric instrumental variables
regression function \citep{neweypowell}. In particular, with binary
instrument and treatment, along with baseline covariates, the latter
includes the average treatment effect estimand under
the model of \cite{wang2018bounded}. In the absence of covariates, this estimand agrees with the parameter that identifies the  local average treatment effect, under the model of
\cite{imbens}.

A substantial body of work proposes methods for constructing $\sqrt{n}$-consistent and asymptotically normal estimators for general estimands of the form \eqref{eq:def_phi} \citep{chernozhukov2023simple,bennett2022,bennett2023}, as well as for specific cases such as the average treatment effect in proximal causal inference \citep{Miao2018,tchetgen2020,kallus, cui2023}. In contrast, while confidence sets that remain uniformly asymptotically valid under models accommodating weak instruments have been extensively studied \citep{staiger1997instrumental,andrews2006optimal,andrews2019weak,ma2023}, their construction for generic functionals of the form \eqref{eq:def_phi} remains unexplored under models $\mathcal{M}$ that rule out conditional independence between $W$ and $Z$ given $X$ but allow arbitrarily weak dependence between them.

A confidence set $C_{n}$ for $\varphi(P)$ is asymptotically uniformly valid over a set of probability laws $\mathcal{M}$ at level at least $1-\alpha$ if
$$
\liminf_{n} \inf_{P\in\mathcal{M}} P\left\lbrace \varphi(P)\in C_{n} \right\rbrace \geq 1-\alpha.
$$
To discard trivial confidence sets, we will also require that there exists at least one law $P\in\mathcal{M}$ such that the confidence set is different from the real line with positive probability under $P$.

For practitioners, using confidence sets that are asymptotically uniform over rich classes of probability laws is important, because, for non-uniform confidence sets, at any finite sample size there exist data generating
distributions under which the coverage of the true parameter is less than the nominal $1-\alpha$.

In Section \ref{sec:main}, we show that Wald confidence intervals are not uniformly valid over weak dependence models. A practical implication is that when the dependence between $W$ and $Z$ given $X$ is suspected to be weak, Wald intervals should be avoided, as they may exhibit poor coverage. While similar results have been established in the Econometrics literature on the behaviour of Wald confidence intervals in linear instrumental variable models with weak instruments \citep{gleser,dufour}, our work is the first to establish the non-uniformity of Wald confidence intervals over weak dependence models for the broad class of estimands \eqref{eq:def_phi}.

We arrive at the non-uniformity of Wald confidence intervals by establishing a necessary condition for a confidence set to be uniformly valid over a weak dependence model $\mathcal{M}$. This condition states that, under some laws in $\mathcal{M}$, the confidence set must have, with high probability, a diameter that is not smaller than the diameter of $\lbrace \varphi(P): P\in\mathcal{M}\rbrace$, a requirement not met by Wald confidence intervals when the parameter's range is infinite, nor by sets with a diameter that converges uniformly to zero when the parameter's range is finite. Crucially, our results hold even when all variables have finite ranges,
and are thus unrelated to the curse of dimensionality, or to slow rates of convergence in estimating possibly non-smooth conditional expectations.  
Our results contribute
to the existing literature on impossibility results in Statistics,
dating back to \cite{kraft1955}. To prove our results, we build on previous work of \cite{gleser,dufour,Romano2004,Canay2013,impossible,Shah2020}. 

% In Section \ref{sec:main}, we show that under regularity conditions on the model $\mathcal{M}$ and the target parameter, there exist laws $P^{\ast} \notin \mathcal{M}$ where $W$ is conditionally independent of $Z$ given $X$, such that for any possible value $\zeta$ of the target parameter, there is a law $P \in \mathcal{M}$ arbitrarily close to $P^{\ast}$ with $\varphi(P) = \zeta$. We further
% show that this result has two important implications: 
% (i) a necessary condition for a confidence set to be asymptotically uniform
%     over $\mathcal{M}$ is that, under some data-generating laws in $\mathcal{M}$, the confidence set must have, with high probability, a diameter that is larger than the diameter of the set of values of  $\varphi
%     \left( P\right)$ as $P$ varies over $\mathcal{M}$,
%     (ii) 
% uniformly consistent estimators of $\varphi \left( P\right) $ under $
%     \mathcal{M}$ do not exist. 
%     An estimator $\widehat{\varphi}_{n}$ is uniformly consistent for $\varphi(P)$ over a set of probability laws $\mathcal{M}$ if 
% $\limsup_{n}\sup_{P\in\mathcal{M}}P\left\lbrace \left\vert \widehat{\varphi}_{n}-\varphi(P) \right\vert\geq \varepsilon\right\rbrace=0
% $ for every $\varepsilon>0$.

A natural first attempt to construct uniformly valid confidence sets for general parameters $\varphi(P)$ is to adapt the approach of \cite{anderson1949, stockwright, ma2023} who build confidence sets for the local average treatment effect estimand with a binary treatment and instrument by inverting the score test. However, in the Section \ref{sec:score} we show that this strategy relies on the specific structure of the local average treatment effect estimand and, perhaps surprisingly, does not extend to the broader class of estimands in \eqref{eq:def_phi}. 
In Section \ref{sec:binary}, we propose a method for constructing uniformly valid confidence sets for general estimands of the form \eqref{eq:def_phi} when all variables, but possibly $Y$, are binary. While this approach can in principle be extended to cases where all variables are finitely valued, we believe it leaves significant room for improvement. 
Thus, while we do not provide a fully satisfactory solution to the problem of constructing uniformly valid confidence sets for general estimands of the form \eqref{eq:def_phi}, we hope our results will encourage further research in this direction.

\section{The model}\label{sec:model}

The results that we will derive assume some regularity conditions on the
model $\mathcal{M}$ and the parameter $\varphi \left( P\right) $.
To start with, we shall assume that the law $P$ of $O$ 
belongs to the set $\mathcal{P}(\mu)$ of
Borel
probability measures over $\mathbb{R} \times \mathbb{R}^{d_{z}} \times \mathbb{R}^{d_{w}}\times\mathbb{R}^{d_{x}}$ that are absolutely continuous with respect to
some Borel measure $\mu$. We shall assume that $\mu $ satisfies the
following condition.

\begin{condition}\label{cond:product_measure}
    $\mu= \mu_{Y} \times \mu_{Z} \times \mu_{W}\times \mu_{X}$, where $\mu_{Y}$ is a Borel $\sigma$-finite measure on $\mathbb{R}$ that is either atomless or a counting measure on a finite set with two or more elements, $\mu_{X}$ is a  Borel $\sigma$-finite measure on $\mathbb{R}^{d_{x}}$ and one of the following holds:
    \begin{enumerate}
        \item $\mu_{Z}=\mu_{Z_{1}}\times \mu_{Z_{2}}$, $\mu_{W}=\mu_{W_{1}}\times \mu_{W_{2}}$, where $\mu_{Z_{1}}$ and $\mu_{W_{1}}$ are atomless Borel $\sigma$-finite measures on $\mathbb{R}^{d_{z_{1}}}$ and $\mathbb{R}^{d_{w_{1}}}$ respectively, and $\mu_{Z_{2}}$ and $\mu_{W_{2}}$ are  Borel $\sigma$-finite measures on $\mathbb{R}^{d_{z_{2}}}$ and $\mathbb{R}^{d_{w_{2}}}$.
        \item $\mu_{Z}$ and $\mu_{W}$  are counting measures on finite sets $\mathcal{Z}\subset \mathbb{R}^{d_{z}}$, $\mathcal{W}\subset \mathbb{R}^{d_{w}}$  and  $\#\mathcal{Z}=\#\mathcal{W}$.
    \end{enumerate}
\end{condition}
Thus our results apply, for example, to random vectors $Z$ and $W$ that each
have at least one continous component, as in such case, part 1 of Condition
\ref{cond:product_measure} holds. They also apply to $Z$ and $W$ discrete random variables with
finite supports of the same cardinality as in such case, part 2 of Condition
\ref{cond:product_measure} holds. 

We shall also assume that $P$ is such that the map 
\begin{equation}
E_{P}\left\lbrace  m\left( O,\cdot \right) \right\rbrace :L^{2}\left\lbrace  P_{(W,X)}\right\rbrace 
\mapsto \mathbb{R}\text{ is linear and continuous}  \label{eq:cont}.
\end{equation}
Next, we require that equation \eqref{eq:integral_eq} has a solution. Formally, if we let $%
T_{P}:L^{2}\left\lbrace P_{(W,X)}\right\rbrace \mapsto L^{2}\left\lbrace P_{(Z,X)}\right\rbrace $ be
defined as the conditional mean operator 
$$
(T_{P}g)(Z,X):=E_{P}\left\lbrace g(W,X)\mid Z,X\right\rbrace
$$
and we let $\range\left( T_{P}\right)$ denote the range of $T_{P},$ we
then require that $\mathcal{M}$ be comprised of laws $P$ such that $r_{P}$ belongs to $\range\left( T_{P}\right) $ where $r_{P}\left( Z,X\right) := E_{P}\left( Y|Z,X\right)
.$ We do not require that the solution to equation \eqref{eq:integral_eq} be unique, but we do
require that $\varphi \left( P\right) $ is well defined, meaning that if $%
g_{p}$ and $g_{P}^{\prime }$ are two solutions of equation \eqref{eq:integral_eq}, then $E_{P}%
\left\lbrace m\left( O,g_{P}\right) \right\rbrace =E_{P}\left\lbrace m\left( O,g_{P}^{\prime
}\right) \right\rbrace .$ This requirement can be formalized as follows. By \eqref{eq:cont}, the Riesz Representer Theorem establishes that
there exists $\alpha _{P}\in L^{2}\left\lbrace P_{(W,X)}\right\rbrace $ such that 
\[
    \varphi(P)= E_{P}\left\lbrace \alpha_{P}(W,X) g_{P}(W,X)\right\rbrace. 
\]
It then follows that $E_{P}\left\lbrace m\left( O,g_{P}\right) \right\rbrace =E_{P}\left\lbrace
m\left( O,g_{P}^{\prime }\right) \right\rbrace$ if and only if $E_{P}\left\lbrace
\alpha_{P}\left( W,X\right) \left\{ g_{P}-g_{P}^{\prime }\right\} \left(
W,X\right) \right\rbrace =0$. Since $g_{P}-g_{P}^{\prime }$ is an arbitrary
element of $\ker \left( T_{P}\right)$, the kernel of the operator $T_{P}$,
we then conclude that $\varphi$ is well defined if and only if $\alpha _{P}$
is in the orthogonal complement of $\ker \left( T_{P}\right)$. This
condition is equivalent to the condition that $\alpha _{P}$ is in the
closure of the range, $\range\left( T_{P}^{\prime }\right) $, of the
adjoint operator 
$T^{\prime}_{P}: L^{2}\lbrace P_{(Z,X)}\rbrace \to L^{2}\lbrace P_{(W,X)}\rbrace$, given by
$$
(T^{\prime}_{P}h)(W,X):=E_{P}\left\lbrace h(Z,X)\mid W,X\right\rbrace,
$$
for $h\in L^{2}\lbrace P_{(Z,X)}\rbrace $.
In fact, we will make
the stronger assumption that further requires that $\alpha _{P}$ is in $%
\range\left( T_{P}^{\prime }\right)$, a strengthening typically assumed in the lireature \citep{severini,Zhang2023,bennett2022}. 

Our model will also accommodate further restrictions on the law $P$. These include two type of restrictions. The first are constraints on the possible values of $\varphi(P)$, through restrictions on the causal estimand it identifies. Specifically, it may happen that it is a priori known that $\varphi(P)$ falls in $\mathcal{S}$, a Borel subset of the real line. For instance, this is the case in Examples \ref{ex:ATE_proximal}, \ref{ex:ATE_IV} and \ref{ex:LATE}  below, if the outcome is a binary random variable. The second are constraints that include
positivity assumptions, which are typically required to derive results on the asymptotic properties of estimators and confidence sets for $\varphi(P)$. Specifically, the second type of restriction, assumes that $P\in\mathcal{R}$, where $\mathcal{R}$ is a subset of $\mathcal{P}(\mu)$ that satisfies the following condition.  In what follows, for $P\in \mathcal{P}(\mu)$, we let $f_{P}$ be the density of $P$ with respect to $\mu$ and $\convas$ denotes $\mu$-almost everywhere convergence.
\begin{condition}\label{cond:R}
    If $P\in\mathcal{R}$ and $(P_{i})_{i\geq 1}$ is a sequence of laws in $\mathcal{P}(\mu)$ such that $f_{P_i}\convas f_{P}$ as $i\to\infty$, then there exists a subsequence of $(P_{i})_{i\geq 1}$ that belongs to $\mathcal{R}$.
\end{condition}
For instance, in Example \ref{ex:LATE} below, we may take $\mathcal{R}=\lbrace P\in \mathcal{P}(\mu): f_{P,Z}(Z) \in (c,1-c) \: , P-\text{almost surely}\rbrace$, for some fixed $c\in (0,1)$, where throughout, if $V$ is a subvector of $O$ and $P\in\mathcal{P}(\mu)$, we let $f_{P,V}$ be the density with respect to $\mu$ of $P_{V}$.

In conclusion, we will assume a model for the law of $O$ defined as
\begin{equation}
\mathcal{M}=\left\{ P\in \mathcal{P}(\mu): P\in\mathcal{R},\eqref{eq:cont} \text{ holds, } \alpha _{P}\in \range\left( T_{P}^{\prime
}\right), r_{P}\in \range\left( T_{P}\right) \text{ and } \varphi(P)\in\mathcal{S}
\right\},
\label{eq:model_M}
\end{equation}
where $\mathcal{R}$ is a subset of $\mathcal{P}(\mu)$ that satisfies Condition \ref{cond:R}.
Note that $\mathcal{M}$ depends on the given target
parameter $\varphi $ through the restrictions that \eqref{eq:cont} holds and the Riesz representer $\alpha _{P}$ is in $\range\left(
T_{P}^{\prime }\right)$, but we do not make this dependence explicit in the notation.

Regarding parameters of the form \eqref{eq:def_phi}, our results are derived for those that
satisfy the regularity conditions specified in Conditions \ref{cond:alpha_coheres} and \ref{cond:alpha_convergence_1} stated in the Appendix. At a high level, Condition \ref{cond:alpha_coheres} requires that, if $f_{P}$ is a special kind of simple function, that is, a function taking only a finite number of values, then \eqref{eq:cont} holds, and $\alpha_{P}$ is also a simple function. Condition \ref{cond:alpha_convergence_1} is a type of continuity assumption on the map $P\mapsto \alpha_{P}$. As
demonstrated in Proposition \ref{prop:examples} in the Appendix, these conditions hold for the parameters in the examples below, which under certain identifying assumptions coincide with estimands of interest in causal inference. 
It is important to note that while these identifying assumptions may impose additional inequality constraints on 
$P$, we do not enforce such constraints in our model 
$\mathcal{M}$, except for those that can expressed as the restriction that $P\in\mathcal{R}$ where $\mathcal{R}$ satisfies Condition \ref{cond:R}, and those that restrict the possible values of 
$\varphi(P)$ through constraints on the causal estimand it identifies.

 In what follows, if $V_{1}$ is a subvector of $O$ and $V_{2}$ is another subvector of $O$ we let $f_{P,V_{1}\mid V_{2}}$ be the conditional density of $V_{1}$ given $V_{2}$.

\begin{example}[Average treatment effect in proximal causal inference]\label{ex:ATE_proximal}
    Suppose $Y$ is an outcome of interest, $Z$ is a treatment inducing proxy, $W$ is an outcome inducing proxy and
    $X=(A,L)$, where $A$ is a binary treatment and $L$ is a vector of covariates. \cite{tchetgen2020,kallus,cui2023} show that, under the assumptions discussed in their papers, the average treatment effect $E\lbrace Y(1)-Y(0)\rbrace$, where $Y(a)$ is the potential outcome under treatment $A=a$, coincides with the parameter $\varphi(P)$, with $m(O,g)=g(W,1,L) - g(W,0,L)$ and $\alpha_{P}(W,A,L)=(2A-1)/f_{P,A\mid(W,L)}(A\mid W,L)$.
\end{example}

\begin{example}[Weighted integral of the non-parametric instrumental variable function]\label{ex:np_IV}
    Suppose $Y$ is an outcome of interest, $Z$ is an instrumental variable, $W$ is a treatment variable and $X$ does not exist.
    \cite{neweypowell} show that, under certain identifying assumptions, $\int E\left\lbrace Y(w) \right\rbrace \omega(w) d\mu_{W}(w)$, where $Y(w)$ is the potential outcome under treatment $W=w$ and $\omega$ is a known function, coincides with the parameter $\varphi(P)$, with $m(O,g)=\int g(w) \omega(w) d\mu_{W}$ and  $\alpha_{P}(W)=\omega(W)/f_{P,W}(W)$. Our results apply to the case in which $\omega$ is a simple function.
\end{example}

\begin{example}[Average treatment effects with instruments]\label{ex:ATE_IV}
    Suppose $Y$ is an outcome of interest, $Z$ is a binary instrumental variable, $W$ is a binary treatment variable and
    $X$ is a vector of covariates. \cite{wang2018bounded} show that, under certain identifying assumptions, the average treatment effect $E\lbrace Y(1)-Y(0)\rbrace$ coincides with the parameter $\varphi(P)$, with $m(O,g)=g(1,X)-g(0,X)$ and  $\alpha_{P}(W,X)=(2W-1)/f_{P,W\mid X}(W\mid X)$.
\end{example}

\begin{example}[Local average treatment effect]\label{ex:LATE}
    Suppose $Y$ is an outcome of interest, $Z$ is a binary instrumental variable, $W$ is a binary treatment variable and
    $X$ does not exist. \cite{imbens} show that, under certain identifying assumptions, the local average treatment effect $E\lbrace Y(1)-Y(0)\mid W(1)>W(0)\rbrace$, where $Y(w)$ is the potential outcome of $Y$ under  $W=w$ and $W(z)$ is the potential outcome of $W$ under $Z=z$, coincides with the parameter $\varphi(P)$, with $m(O,g)=g(1)-g(0)$ and $\alpha_{P}(W)=(2W-1)/f_{P,W}(W)$. Thus $\varphi(P)$ coincides with the target parameter identifying the average treatment effect in Example \ref{ex:ATE_IV}, in the absence of covariates. We note that the function of $P$ that identifies the local average treatment effect in the presence of covariates under the model of \cite{imbens} does not agree with the target parameter in Example \ref{ex:ATE_IV} and furthermore is not of the form \eqref{eq:def_phi}.
\end{example}

\section{Main results}\label{sec:main}

We are now ready to state the main results of the paper. 
In words, Theorem \ref{theo:main} below establishes that, when $\mathcal{S}=\mathbb{R}$, there exist laws $P^{\ast} \notin \mathcal{M}$ such that for any value $\zeta\in\mathbb{R}$, there is a law $P \in \mathcal{M}$ arbitrarily close to $P^{\ast}$ in total variation, with $\varphi(P) = \zeta$.

Theorem \ref{theo:main} below, uses the total variation distance 
between two probability laws $P_{1}$ and $P_{2}$, defined as $\Vert P_{1} - P_{2} \Vert_{TV}:=\sup\limits_{B \: \text{Borel}} \vert P_{1}(B) -P_{2}(B)\vert$.

\begin{theorem}\label{theo:main}
    Assume that Condition \ref{cond:product_measure} and Conditions \ref{cond:alpha_coheres} and \ref{cond:alpha_convergence_1} in the Appendix hold and that $\mathcal{S}=\mathbb{R}$. Then, there exists $P^{\ast}\notin\mathcal{M}$ such that for any $\zeta \in \mathbb{R}$ there exists a sequence $(P_{i})_{i\geq 1}$ of laws in $\mathcal{M}$ satisfying
    \begin{enumerate}
        \item $\Vert P^{\ast} -P_{i}\Vert_{TV} \to 0$ as $i\to \infty$.
        \item $\varphi(P_{i})=\zeta$ for all $i$.
    \end{enumerate}
\end{theorem}

Invoking Theorem 3 of \cite{impossible}, we arrive at  the following corollary of Theorem \ref{theo:main}. 
\begin{corollary}\label{coro:confidence_impossible}
     Let $C_{n}=C_{n}(O_{1},\dots,O_{n})$ be an asymptotically uniformly valid confidence set of level at least $1-\alpha$ for $\varphi(P)$ over $\mathcal{M}$.
     Assume that $diam(\mathcal{S})=+\infty$.
    It holds that
    \[ 
    \liminf_{n}\sup\limits_{\mathcal{M}} P\left( \text{diam}(C_{n})=+\infty \right) \geq 1-\alpha,
    \]
    where $\text{diam}(C_{n}):=\sup \lbrace x-y : x,y \in C_{n}\rbrace $ is the diameter of $C_{n}$.

\end{corollary}

Corollary \ref{coro:confidence_impossible} establishes that, when the range of $\varphi(P)$ has infinite diameter, for any asymptotically uniformly valid confidence set in model $\mathcal{M}$, there exist laws in $\mathcal{M}$ at which the confidence set has, with high probability, an infinite diameter. An immediate consequence of Corollary \ref{coro:confidence_impossible} is that Wald confidence intervals for $\varphi(P)$ cannot be uniformly valid in model $\mathcal{M}$ when $diam(\mathcal{S})=+\infty$, because Wald confidence intervals have finite length with probability one at each $P\in\mathcal{M}$.

In the next corollary, we provide a refinement of Corollary \ref{coro:confidence_impossible} that applies even to models in which the range of $\varphi(P)$ is restricted.

\begin{corollary}\label{coro:confidence}
     Let $C_{n}=C_{n}(O_{1},\dots,O_{n})$ be an asymptotically uniformly valid confidence set of level at least $1-\alpha$ for $\varphi(P)$ over $\mathcal{M}$. Then
    $
    \liminf_{n}\sup_{P\in \mathcal{M}} P\left\lbrace  diam(C_{n}) \geq diam(\mathcal{S}) \right\rbrace \geq 1-2\alpha.
    $
    In particular, if $C_{n}$ is a set that satisfies that for every $\varepsilon>0$ it holds that $\lim_{n}\sup_{P\in \mathcal{M}} P\left\lbrace diam(C_{n}) \geq \varepsilon \right\rbrace = 0$, then $C_{n}$ is not a uniformly valid confidence set.
\end{corollary}

Corollary \ref{coro:confidence} establishes that for any asymptotically uniformly valid confidence set, for all sufficiently large $n$, there exist laws in $\mathcal{M}$ at which the confidence set has a diameter that is greater than the diameter of $\mathcal{S}$, with high probability.
This implies that if $C_{n}$ is a set with diameter that converges to zero in probability, uniformly over $\mathcal{M}$, then $C_{n}$ cannot be a uniformly valid confidence set.

We also have the following proposition. 
\begin{proposition}\label{coro:minimax} 
    There exists a constant $c>0$ such that for all $n$ and all estimators $\widehat{\varphi}_{n}=\varphi(O_{1},\dots,O_{n})$,
    \begin{equation}
        \sup\limits_{P\in \mathcal{M}} d_{P}\lbrace \widehat{\varphi}_{n},\varphi(P)\rbrace
          \geq c
          \nonumber
    \end{equation}
    where $d_{P}(U,V):=E_{P}\left\lbrace \min(\vert U-V\vert,1)\right\rbrace$.
\end{proposition}
Since the metric $d_{P}$ metrizies convergence in probability under $P$, that is $U_n - V_{n}\to 0$ in probability under $P$ if and only $d_{P}(U_{n},V_{n})\to 0$, 
Proposition \ref{coro:minimax} implies that there do not exist uniformly consistent estimators of $\varphi(P)$ over $\mathcal{M}$.

\section{Inverting the score test}\label{sec:score}

As noted in Example \ref{ex:LATE}, the local average treatment effect in a binary instrumental variable model without covariates is identified by a functional of the form \eqref{eq:def_phi}. Specifically, under the assumptions in \cite{imbens}, the local average treatment effect is equal to the Wald estimand 
$$
\varphi(P)=\frac{E_{P}\left(Y\mid Z=1\right)-E_{P}\left(Y\mid Z=0\right)}{E_{P}\left(W\mid Z=1\right)-E_{P}\left(W\mid Z=0\right)},
$$
which corresponds to $m(O,g)=g(1)-g(0)$ and $\alpha_{P}(W)=(2W-1)/f_{P,W}(W)$. Recall from Example \ref{ex:LATE} that this estimand also identifies the average treatment effect in the model of \cite{wang2018bounded} in the absence of covariates.

Suppose we are interested in constructing a uniform confidence set for $\varphi(P)$ in a model $\mathcal{M}$ in which $\mathcal{R}$ is defined  by the restriction that $c<f_{P,Z}(Z=1)<1-c$ for some $c>0$ and $\mathcal{S}$ is an open bounded subset of $\mathbb{R}$.
Following the well established strategy in the weak instruments literature \citep{anderson1949,andrews2019weak,ma2023} we can construct such confidence set by inverting the score test. 
Specifically, a level $1-\alpha$ uniformly asymptotically valid confidence set is given by
$$
C_{n}=\lbrace \theta: \vert \mathcal{T}_{n}(\theta)\vert \leq \zeta_{1-\alpha/2}\rbrace
$$
where
\begin{align}
    \mathcal{T}_{n}(\theta)= \frac{\sqrt{n}\mathbb{P}_{n}\left\lbrace \eta^{1}_{\mathbb{P}_{n},\theta}(O)\right\rbrace}{\sigma_{\mathbb{P}_{n},\theta}},
    \nonumber
\end{align}
$\mathbb{P}_{n}$ is the empirical distribution, for any law $P$ 
\begin{align*}
 &\eta^{1}_{P,\theta}(O):=\frac{2Z-1}{f_{P,Z}(Z)}\frac{1}{cov_{P}(W,Z)}\left[ Y-E_{P}(Y\mid Z=1)- \theta \lbrace W- E_{P}(W\mid Z=1)\rbrace \right],
 \\
 &\sigma^{2}_{P,\theta}:= E_{P}\left[\left\lbrace \eta^{1}_{P,\theta}(O)\right\rbrace^{2}\right],
\end{align*}
and $\zeta_{1-\alpha/2}$ is the $1-\alpha/2$ quantile of the standard normal distribution. See \cite{ma2023}. The random variable $\eta^{1}_{P,\theta}(O)$ in an unbiased estimating function for $\varphi(P)$, in that $E_{P}\left\lbrace \eta^{1}_{P,\theta}(O)\right\rbrace=0$ if and only if $\theta=\varphi(P)$.

The confidence set $C_{n}$  is uniformly asymptotically valid because the distribution of the test statistic $\mathcal{T}_{n}\left\lbrace\varphi(P)\right\rbrace$ converges to a standard normal distribution uniformly over model $\mathcal{M}$. Note that the term $cov_{P}(W,Z)$ appears as a denominator in $\eta^{1}_{P,\theta}$, and could be arbitrarily small for some $P\in\mathcal{M}$. This could cause the random variable $\eta^{1}_{\mathbb{P}_{n},\theta}$ to behave erratically. Nevertheless, $\mathcal{T}_{n}\left\lbrace\varphi(P)\right\rbrace$ is uniformly asymptotically normal because, the term $cov_{\mathbb{P}_{n}}(W,Z)$ in  $\eta^{1}_{\mathbb{P}_{n},\theta}$ cancels out with the one apearing in $\sigma_{\mathbb{P}_{n},\theta}$, and the remaining nuisance parameters that appear in $\eta^{1}_{P,\theta}$, namely $f_{P,Z}(Z),E_{P}(Y\mid Z=1)$ and $E_{P}(W\mid Z=1)$ can be estimated uniformly at a rate $\sqrt{n}$ in model $\mathcal{M}$.

For an arbitrary parameter of the form \eqref{eq:def_phi}, an unbiased estimating function is  $\psi^{1}_{g_{P},q_{P},\theta}(O)$ where $q_{P}$ is a solution to the integral equation 
$
E_{P}\left\lbrace q(Z,X)\mid W,X\right\rbrace=\alpha_{P}(W,X)
$
and for any $g,q$ and $\theta$,
\begin{equation}
    \psi^{1}_{g,q,\theta}(O):=m(O,g)+q(Z,X)\left\lbrace Y-g(W,X)\right\rbrace - \theta.
    \label{eq:psi1_new}
    \end{equation} 
See \cite{bennett2022}. For the special case of the local average treatment effect without covariates $X$, $\psi^{1}_{g_{P},q_{P},\varphi(P)}$ coincides with $\eta^{1}_{P,\varphi(P)}$. To see this, notice that: (i)
$m(O,g_{P})$ cancels with $\theta=\varphi(P)$ because $m(O,g_{P})=g_{P}(1)-g_{P}(0)=\varphi(P)$, (ii)
    $g_{P}(W)=E_{P}(Y\mid Z=1)+ \varphi(P) \lbrace W- E_{P}(W\mid Z=1)\rbrace $ and  $q_{P}(Z)=\frac{2Z-1}{f_{P,Z}(Z)} \frac{1}{cov_{P}(W,Z)}$.

Given the success of the strategy of inverting the score test for the local average treatment effect, a natural attempt for extending it to the general case would be to consider the following set, obtained by inverting a score test based on the general form of the estimating equation \eqref{eq:psi1_new}: 
$$
D_{n}=\lbrace \theta: \vert \mathcal{W}_{n}(\theta)\vert \leq \zeta_{1-\alpha/2}\rbrace
$$
 where
\begin{align}
    \mathcal{W}_{n}(\theta):= \frac{\sqrt{n}\mathbb{P}_{n}\left\lbrace \psi^{1}_{\widehat{g},\widehat{q},\theta}(O)\right\rbrace}{\widehat{\sigma}_{\widehat{g},\widehat{q},\theta}},
    \nonumber
\end{align}
$\widehat{g},\widehat{q}$ are estimators of $g_{p}$ and $q_{P}$, possibly computed in an independent sample, 
and $\widehat{\sigma}_{\widehat{g},\widehat{q},\varphi(P)}$ converges in probability, uniformly over $\mathcal{M}$, to 
$$
\sigma_{P}:=E^{1/2}_{P}\left[ \left\lbrace \psi^{1}_{g_{P},q_{P},\varphi(P)}(O)\right\rbrace^{2}\right].
$$
Note that this strategy does not leverage any  possible special simplifications in the estimating equation, such as 1) and 2) above for the local average treatment effect.

Unfortunately, $D_{n}$ is not uniformly asymptotically valid in model $\mathcal{M}$, even for the local average treatment effect. This happens because while for some $P$ in $\mathcal{M}$, and for suitably chosen $\widehat{g}$ and $\widehat{q}$, it holds that 
$\mathcal{W}_{n}\lbrace \varphi(P)\rbrace$ converges in distribution to a standard normal, the following proposition establishes that this convergence is not uniform over $\mathcal{M}$.
The proposition is stated in terms of the Levy-Prokhorov distance, $d_{LP}$, defined for any two probability laws $P_{1}$ and $P_{2}$ as
$$
d_{LP}(P_{1},P_{2}):= \inf\lbrace \varepsilon>0: P_{1}(B)\leq P_{2}(B^{\varepsilon}) + \varepsilon \text{ for all } B \text{ Borel}\rbrace,
$$ 
where if $B$ is a Borel set in the real line,
$$
   B^{\varepsilon}:=\lbrace t\in\mathbb{R}: \text{there exists }s\in B \text{ such that }\vert s-t\vert<\varepsilon \rbrace.
$$
The distance $d_{LP}$ metrizies converges in distribution \citep{billingsley2013}, that is, a sequence of laws $P_{n}$ converges to a law $P$ in distribution if and only if $d_{LP}(P_{n},P)\to 0$ as $n\to\infty$.

\begin{proposition}\label{coro:no_score_test}
    Let $O_{1},\dots,O_{n}$ be independent identically distributed random variables with law $P\in\mathcal{M}$.
    Let $(\widehat{g}_{n})_{n\geq 1}$ and  $(\widehat{q}_{n})_{n\geq 1}$ be sequences of functions such that $\widehat{g}_{n}$ is a function of $(W,X)$ and  $\widehat{q}_{n}$ is a function of $(Z,X)$ for each $n$. Moreover, assume that, for all $n$, $\widehat{g}_{n}(w,x)$ and $\widehat{q}_{n}(z,x)$ are statistics for every $w,z$ and $x$ in the ranges of $W,Z$ and $X$ respectively, that is, $\widehat{g}_{n}(w,x)$ and $\widehat{q}_{n}(z,x)$ depend on $O_{1},\dots,O_{n}$  but not on $P$.
   Assume that $\widehat{\sigma}_{\widehat{g},\widehat{q},\varphi(P)}$ converges in probability, uniformly over $\mathcal{M}$, to $\sigma_{P}$ and that $\sup_{P\in\mathcal{M}} \sigma_{P}<\infty$. 
   Let $P_{{\mathcal{W}}_{n}}$ be the distribution of ${\mathcal{W}}_{n}\lbrace \varphi(P) \rbrace$ when $O$ has distribution $P$ and let $\Phi$ be the distribution of the standard normal. Then,
   $
   \limsup_{n\to\infty} \sup_{P\in\mathcal{M}}d_{LP}(P_{{\mathcal{W}}_{n}},\Phi)>0.
   $
\end{proposition}

Having established that $D_{n}$ is not a uniformly asymptotically valid confidence set, 
we have attempted to find alternative representations of the estimating equation \eqref{eq:psi1_new} that could lead to a score test which yields a uniformly valid confidence set for estimands different from the local average treatment effect, in the special case in which all variables are binary. We have not been able to find such a representation.

In concluding this section, we highlight that although the results derived here and in the preceding sections assume a model $\mathcal{M}$ of the form \eqref{eq:model_M}, inspecting the proofs of our results shows that these remain valid even when $\mathcal{M}$ is redefined to additionally include the additional restriction $\alpha_{P} \in \range(T_{P}^{\prime}T_{P})$, made often in the literature \citep{bennett2022,bennett2023}.

In the next section, we discuss an alternative general strategy for constructing uniformly valid confidence sets for an arbitrary parameter of the form \eqref{eq:def_phi} when all variables, except for possibly $Y$, are binary.

\section{A uniformly valid confidence set for the special case of binary variables}\label{sec:binary}

Throughout this section, we will assume that $Y$ is a bounded random variable and  all remaining variables are binary. In this case, it can be shown that
\begin{align}
    &g_{P}(W,X)= \frac{g^{nu,X}_{P}}{g^{de,X}_{P}} \left\lbrace W - E_{P}(W\mid Z=1,X)\right\rbrace  + E_{P}(Y\mid Z=1,X),
    \label{eq:formula_g_binary}
\end{align}
where 
\begin{align*}
    &g_{P}^{nu,x}:= E_{P}(Y\mid Z=1, X=x)-E_{P}(Y\mid Z=0, X=x),
    \\
    &g_{P}^{de,x}:= E_{P}(W\mid Z=1, X=x)-E_{P}(W\mid Z=0, X=x).
\end{align*}
We illustrate a strategy for constructing uniformly valid confidence sets for $\varphi(P)$ when $m(O,g)=g(W,1)$. In such case, the estimand $\varphi(P)$, which coincides with the counterfactual mean under treatment $X=1$ without baseline covariates in the proximal causal inference framework, can be expressed as
$$
\varphi(P)=\frac{g^{nu,1}_{P}}{g^{de,1}_{P}} \left\lbrace E_{P}(W) - E_{P}(W\mid Z=1,X=1)\right\rbrace  + E_{P}(Y\mid Z=1,X=1).
$$
We thus see that $\varphi(P)$ can be expressed in terms of sums, products and ratios of conditional and unconditional expectations. 
The strategy for constructing a uniformly valid confidence set for $\varphi(P)$ now goes as follows.
\begin{enumerate}
    \item Construct uniformly asymptotically valid confidence set for $g_{P}^{de,0}$, $g_{P}^{de,1}$ and all the other expectations or difference of expectations that appear in the expression for $\varphi(P)$. All of these confidence sets can be, for example, Wald confidence intervals.
    \item The confidence set for $\varphi(P)$ is obtained by piecing together the individual confidence sets we obtained in the first step, by summing, multiplying and dividing them, replicating the operations that appear in the expression for $\varphi(P)$.
\end{enumerate}
Specifically, let
\begin{align*}
    &\gamma^{W}_{P}:= g^{nu,1}_{P} \left\lbrace E_{P}(W) - E_{P}(W\mid Z=1,X=1)\right\rbrace 
    \\
    &\gamma^{Y}_{P}:= E_{P}(Y\mid Z=1,X=1),
\end{align*}
and let $B_{n}^{de},B_{n}^{W},B_{n}^{Y}$ be asymptotically uniformly valid confidence sets over $\mathcal{M}$ of level $1-\alpha/3$ for $g^{de,1}_{P},\gamma^{W}_{P}$ and $\gamma^{Y}_{P}$ respectively. For instance, $B_{n}^{de},B_{n}^{W},B_{n}^{Y}$ can be Wald confidence intervals. The proposed confidence set for $\varphi(P)$ is 
$$
B_{n}:=\lbrace s/t + r : s\in B_{n}^{W}, t\in B_{n}^{de} \setminus \lbrace 0 \rbrace, r\in B_{n}^{Y} \rbrace \cap \mathcal{S},
$$
where, recall, $\mathcal{S}$ is the set of possible values of $\varphi(P)$ over the model. Note that if $B_{n}^{de}$ contains an interval centered at zero and $B_{n}^{W}$ contains at least one non-zero value, then $B_{n}=\mathcal{S}$. The set $B_{n}$ is a uniformly asymptotically valid confidence set of level at least $1-\alpha$. This follows because, by the union bound, we have that
$$
\liminf_{n} \inf_{P\in\mathcal{M}} P\left\lbrace g^{de,1}_{P}\in B_{n}^{de},  \gamma^{W}_{P}\in B_{n}^{W},\gamma^{Y}_{P}\in B_{n}^{y} \right\rbrace \geq 1-\alpha.
$$
Now, for all $P\in\mathcal{M}$
\begin{align*}
P\left\lbrace \varphi(P) \in  B_{n}\right\rbrace &\geq P\left\lbrace g^{de,1}_{P}\in B_{n}^{de},  \gamma^{W}_{P}\in B_{n}^{W},\gamma^{Y}_{P}\in B_{n}^{Y} \right\rbrace.
\end{align*}
Consequently
\begin{align*}
    \liminf_{n} \inf_{P\in\mathcal{M}}P\left\lbrace \varphi(P) \in  B_{n}\right\rbrace \geq \liminf_{n} \inf_{P\in\mathcal{M}}P\left\lbrace g^{de,1}_{P}\in B_{n}^{de},  \gamma^{W}_{P}\in B_{n}^{W},\gamma^{Y}_{P}\in B_{n}^{Y} \right\rbrace \geq 1-\alpha.
\end{align*}

This strategy has the following shortcomings:
\begin{enumerate}
    \item Since the union bound is generally conservative, the confidence set is expected to be overly so. Moreover, the optimal way to allocate the $\alpha$ level among the different confidence sets is unclear.
    \item There are multiple ways to group the terms in the expression for $\varphi(P)$, and the best choice is uncertain. For example, one could construct separate confidence sets for $g^{nu,1}_{P}$ and $$E_{P}(W) - E_{P}(W\mid Z=1,A=1)$$ and then multiply them.
    \item The estimand has alternative expressions that depend on $q_{P}$, but it is unclear which yields better confidence sets.
\end{enumerate}

\section{Discussion}

A key takeaway of this paper is that Wald confidence intervals are not asymptotically uniformly valid over models that permit arbitrarily weak dependence between $W$ and $Z$ given $X$ when the parameter's range is infinite, with the practical consequence that they should be avoided when this dependence is suspected to be weak. Moreover, naively extending the score test inversion strategy, which works for the local average treatment effect, does not necessarily yield a uniformly valid confidence set for general estimands of the form \eqref{eq:def_phi}.

We provided a strategy for constructing uniformly valid confidence sets for $\varphi(P)$ in the special case where all variables are binary. This strategy, illustrated for $m(O, g) = g(W,1)$, extends to general estimands satisfying \eqref{eq:def_phi}, for variables $Z,W,X$ taking values on finite sets. However,the drawbacks discussed in the preceding section remain. Moreover, we were unable to come up with a strategy for constructing uniformly valid confidence sets in the general case, when none of the variables are required to be discrete.  
We hope our results encourage further research into the construction of uniformly valid confidence set for the general estimands considered here.
\newpage
\section{Appendix}

\subsection{Regularity conditions}

In this section of the Appendix, we state the regularity conditions assumed on $\varphi(P)$. To do so, we first introduce the following notation and definitions.

Definition \ref{def:simple} is a strenghtening of the well known measure-theoretic concept of a positive simple function, requiring that the support sets are Cartesian products of Borel sets in each coordinate.

\begin{definition}\label{def:simple}
    We will say that a measurable function $\pi:\mathbb{R} \times \mathbb{R}^{d_{z}} \times \mathbb{R}^{d_{w}}\times \mathbb{R}^{d_{x}}\to \mathbb{R}$ is a rectangular non-negative simple function if it is equal $\mu$-almost everywhere to a function of the form
    \begin{align}
        (y,z,w,x) \mapsto \sum\limits_{h=1}^{k_{Y}}\sum\limits_{l=1}^{k_{Z}}
        \sum\limits_{j=1}^{k_{W}}\sum\limits_{j=1}^{k_{X}} \pi_{h,l,j,m} I\lbrace (y,z,w,x)\in (S^{Y}_{h}\times S^{Z}_{l} \times   S^{W}_{j} \times S^{X}_{m}) \rbrace,
        \label{eq:simple}
    \end{align}
    where each factor of $S^{Y}_{h}\times S^{Z}_{l} \times   S^{W}_{j} \times S^{X}_{m}$ is a Borel set 
    with finite and positive $\mu$ measure, 
    $(S^{Y}_{h}\times S^{Z}_{l} \times   S^{W}_{j} \times S^{X}_{m}) \cap (S^{Y}_{h^{\prime}}\times S^{Z}_{l^{\prime}} \times   S^{W}_{j^{\prime}} \times S^{X}_{m^{\prime}})= \emptyset$ if  $(h,l,j,m)\neq (h^{\prime},l^{\prime},j^{\prime},m^{\prime})$,
    and $\pi_{h,l,j,m}>0$. The sets $S^{Y}_{h}, S^{Z}_{l},   S^{W}_{j}, S^{X}_{m}$ are called the support sets of the function $\pi$. 
    If any of the variables $Y,Z,W$ or $X$ has a finite range, we assume that its  support sets are singletons, each containing an element from the variable's range, and their union forms the variable's support.
\end{definition}

The first regularity condition on $\varphi(P)$ essentially requires that
if $P$ has a density that is a rectangular non-negative simple function then \eqref{eq:cont} holds and the Riesz representer $\alpha_{P}$ corresponding to the linear functional defined in \eqref{eq:cont}
is a simple function with the same support sets as the density of $P$.
Condition \ref{cond:alpha_coheres} below formalizes this requirement.

\begin{condition}\label{cond:alpha_coheres}

    If $P\in \mathcal{P}\left( \mu \right) $ has density $%
f_{P}$ that is $\mu$-almost everywhere equal to a
rectangular non-negative simple function then the map
\[
E_{P}\left\lbrace m\left( O;\cdot\right) \right\rbrace : L^{2}\left\lbrace P_{(W,X)}\right\rbrace 
\to \mathbb{R}
\]%
is continuous and linear. Furthermore, its Riesz representer $\alpha_{P}$ coincides $P$-almost surely with a function of the form
    \begin{equation}
        (w,x)\mapsto \sum\limits_{j=1}^{k_{W}}\sum\limits_{m=1}^{k_{X}} \alpha_{j,m}  I\lbrace (w,x)\in (S^{W}_{j}\times S^{X}_{m})\rbrace  ,
        \nonumber
    \end{equation}
    for some constants $\alpha_{j,m}$.
\end{condition}

The second regularity condition on $\varphi(P)$ we impose require that the Riesz representer map $P\mapsto \alpha_{P}$ satisfies the following continuity condition.

\begin{condition}\label{cond:alpha_convergence_1}
    Let $P$ be a law satisfying \eqref{eq:cont}. Assume that $(P_{i})_{i\geq 1}$ is a sequence of laws that satisfy \eqref{eq:cont} and such that $f_{P_{i}}\to f_{P}$ as $i\to\infty$, $\mu$-almost everywhere. Then there exists a subsequence $(P_{i_{q}})_{q\geq 1}$ such that $\alpha_{P_{i_{q}}}(W,X)\to \alpha_{P}(W,X)$ almost surely under $P$.
    \end{condition}

The laws $P^{\ast}$ that appear in Theorem \ref{theo:main} satisfy the following condition.
\begin{condition}\label{cond:Past}
    $P^{\ast}\in\mathcal{R}$, $f_{P^{\ast},X}>0$ $\mu$-almost everywhere, and under $P^{\ast}$ it holds that (i) $(Y,W)$ is independent of $Z$ given $X$, (ii) $Y$ is independent of $W$ given $X$, (iii) Equation \eqref{eq:cont} is valid and (iv) $var_{P^{\ast}}\left\lbrace \alpha_{P^{\ast}}(W,X) \mid X \right\rbrace >0$ with positive probability.
\end{condition}

In the following section, we prove Theorem \ref{theo:main}. At a high level, the idea of the proof is to approximate $P^{\ast}$, a law that satisfies Condition \ref{cond:Past}, by a law $\widetilde{P}$ that has a density that is a rectangular non-negative simple function of the form \eqref{eq:simple}, where $k_{Z}=k_{W}$. This essentially reduces the problem to the discrete case where $Z$ and $W$ have the same number of levels. We will then approximate $\widetilde{P}$ by another law $P$. This latter law will be built by adding perturbations to the matrices which encode the conditional laws of $W\mid Z,X$ and $Y\mid Z,X$ under $\widetilde{P}$, so that the matrix that encodes the law of $W\mid Z,X$ under $P$ is invertible. The last step consists on obtaining a closed form expression for the estimand under $P$, and determining how to choose the perturbation factors so that the estimand takes the value $\zeta$, for any given $\zeta\in\mathbb{R}$. We note that, under the approximating law $P$, it holds that $Y$ is independent of $W$ given $X$. We believe that it should be possible to construct the approximating law $P$ without this additional conditional independence, but at the cost of substantially complicating the proof.

Auxiliary results that are needed in the proof are stated and proven in Section \ref{sec:aux}.
We will use the following notation throughout the proof.
We denote the norm in the space $L^{p}(\nu)$ by $\Vert \cdot \Vert_{L^{p}(\nu)}$. We denote convergence in  $\Vert \cdot \Vert_{L^{p}(\nu)}$ norm by $\overset{L^{p}(\nu)}{\to}$, convergence in $\nu$ measure by $\overset{\nu}{\to}$ and uniform convergence by $\overset{u}{\to}$. Given a function $f:\mathbb{R} \times \mathbb{R}^{d_{z}} \times \mathbb{R}^{d_{w}}\times \mathbb{R}^{d_{x}}\to \mathbb{R}$ we let 
$\Vert  f \Vert_{\infty}=\sup\limits_{(y,z,w,x) \in \mathbb{R} \times \mathbb{R}^{d_{z}} \times \mathbb{R}^{d_{w}} \times \mathbb{R}^{d_{x}} } \vert f(y,z,w,x)  \vert$. If $v\in\mathbb{R}^{k}$, we let $\Vert v \Vert$ be its Euclidean norm. 
We will repeatedly use that fact that if $P_{1}, P_{2} \in \mathcal{P}(\mu)$ then $\Vert P_{1} - P_{2} \Vert_{TV} \leq \Vert f_{P_{1}}- f_{P_{2}}\Vert_{L^{1}(\mu)}$.

\subsection{Proof of Theorem \ref{theo:main}}

Take $P^{\ast}$ that satisfies Condition \ref{cond:Past} and fix $\zeta\in \mathbb{R}$.

\subsubsection*{Approximating $P^{\ast}$ by a law with a density that is a rectangular non-negative simple function}

By Lemma \ref{lemma:approximating_density} for each $i\in\mathbb{N}$, there exist
\begin{enumerate}
\item four collections of pairwise disjoint Borel sets 
\begin{align*}
    &  \lbrace S^{Y,i}_{h}: h\in \lbrace 1,\dots,k^{i}_{Y}  \rbrace\rbrace \lbrace S^{Z,i}_{l}: l\in \lbrace 1,\dots,k^{i}  \rbrace\rbrace, \lbrace S^{W,i}_{j}: j\in \lbrace 1,\dots, k^{i} \rbrace \rbrace, \lbrace S^{X,i}_{m}: m\in \lbrace 1,\dots, k^{i}_{X} \rbrace \rbrace
\end{align*} 
such that $0<\mu_{Y}(S^{Y,i}_{h})<\infty$ for all $h\in \lbrace 1,\dots, k^{i}_{Y} \rbrace$, $0<\mu_{Z}(S^{Z,i}_{l})<\infty$ for all $l\in \lbrace 1,\dots, k^{i} \rbrace$,
$0<\mu_{W}(S^{W,i}_{j})<\infty$ for all $j\in \lbrace 1,\dots, k^{i} \rbrace$,
$0<\mu_{X}(S^{X,i}_{m})<\infty$ for all $m\in \lbrace 1,\dots, k^{i}_{X} \rbrace$,$k^{i}\geq 2$, $k_{Y}^{i}\geq 2$
and
\item a set of positive numbers 
\begin{align*}
    % &\lbrace \pi^{i}_{Y,h}: h\in \lbrace 1,\dots,k^{i}_{Y}  \rbrace\rbrace,
    \lbrace \pi^{i}_{(Y,W,Z,X),h,j,l,m}: h\in \lbrace 1,\dots,k^{i}_{Y}  \rbrace, j\in \lbrace 1,\dots,k^{i}\rbrace, m\in \lbrace 1,\dots,k^{i}_{X}  \rbrace, l\in \lbrace 1,\dots, k^{i} \rbrace \rbrace,
\end{align*} 
\end{enumerate}
such that the function
\begin{align*}
    % &\widetilde{f}^{i}_{Y}(y)=\sum\limits_{h} \pi^{i}_{Y,h} I\lbrace y\in S^{Y,i}_{h}\rbrace, 
    % \\
    &\widetilde{f}^{i}(y,w,z,x)=\sum\limits_{h,j,l,m} \pi^{i}_{(Y,W,Z,X),h,j,l,m} I\lbrace y\in S^{Y,i}_{h}\rbrace I\lbrace w\in S^{W,i}_{j}\rbrace I\lbrace z\in S^{Z,i}_{l}\rbrace I\lbrace x\in S^{X,i}_{m}\rbrace,
\end{align*}
satisfies that it is non-negative, integrates to 1 under $\mu$, and 
\begin{align*}
    % &\widetilde{f}^{i}_{Y} \overset{\text{a.e. }\mu}{\to} f_{P^\ast, Y},
    % \\
    &\widetilde{f}^{i} \overset{\text{a.e. }\mu}{\to} f_{P^\ast},
\end{align*} 
which by Scheffe's lemma implies 
\begin{align}
    &\widetilde{f}^{i} \convlp f_{P^\ast}.
    \label{eq:L1conv_tostar_pre}
\end{align} 
Moreover, we can choose the sets $S^{Y,i}_{h}$ such that the vector with coordinates$\int_{S^{Y,i}_{h}} y \: d\mu_{Y}$, $h=1,\dots,k^{i}_{Y}$ is not collinear with the vector of coordinates $\mu_{Y}(S^{Y,i}_{h})$, $h=1,\dots,k^{i}_{Y}$.

Next, we derive expressions for the marginals of $\widetilde{f}^{i}$, which will be useful later on in the proof. Note that
\begin{align*}
    \widetilde{f}^{i}_{(W, Z,X)}(w,z,x)
    &= \int \widetilde{f}^{i}(y,w,z,x) d\mu_{Y} 
    \\
    &= \sum_{j,l,m} \sum_{h} \pi^{i}_{(Y,W,Z,X),h,j,l,m} \mu_{Y}(S^{Y,i}_{h})   I\lbrace w\in S^{W,i}_{j}\rbrace I\lbrace z\in S^{Z,i}_{l}\rbrace  I\lbrace x\in S^{X,i}_{m}\rbrace,
\end{align*}
and
\begin{align*}
    \widetilde{f}^{i}_{(Y, Z,X)}(y,z,x)
    &= \int \widetilde{f}^{i}(y,w,z,x) d\mu_{W} 
    \\
    &= \sum_{h,l,m} \sum_{j} \pi^{i}_{(Y,W,Z,X),h,j,l,m} \mu_{W}(S^{W,i}_{j})   I\lbrace Y\in S^{Y,i}_{h}\rbrace I\lbrace z\in S^{Z,i}_{l}\rbrace  I\lbrace x\in S^{X,i}_{m}\rbrace.
\end{align*}

Let 
\begin{align*}
    &\pi^{i}_{(W,Z,X),j,l,m} = \sum_{h}\pi^{i}_{(Y,W,Z,X),h,j,l,m} \mu_{Y}(S^{Y,i}_{h}),
    \\
    &\pi^{i}_{(Y,Z,X),h,l,m} = \sum_{j}\pi^{i}_{(Y,W,Z,X),h,j,l,m} \mu_{W}(S^{W,i}_{j}).
    % \\
    % &\pi^{i}_{(Y,W,X),h,l,m} = \sum_{l}\pi^{i}_{(Y,W,Z,X),h,j,l,m} \mu_{Z}(S^{Z,i}_{l})
\end{align*}
Then
\begin{align*}
    \widetilde{f}^{i}_{(W, X)}(w, x)
    &= \int \widetilde{f}^{i}_{(W,Z,X)}(w,z,x) d\mu_{Z} 
    \\
    &= \sum_{j,m} \sum_{l} \pi^{i}_{(W,Z,X),j,l,m} \mu_{Z}(S^{Z,i}_{l})   I\lbrace w\in S^{W,i}_{j}\rbrace I\lbrace x\in S^{X,i}_{m}\rbrace
\end{align*}
and
\begin{align*}
    \widetilde{f}^{i}_{X}(x)
    = \int \widetilde{f}^{i}_{(W,X)}(w,x) d\mu_{W} = \sum_{m} \sum_{j,l} \pi^{i}_{(W,Z,X),j,l,m} \mu_{Z}(S^{Z,i}_{l}) \mu_{W}(S^{W,i}_{j})  I\lbrace x\in S^{X,i}_{m}\rbrace.
\end{align*}
Hence, when $x\in S^{X,i}_{m}$,
\begin{align*}
    \widetilde{f}^{i}_{W\mid X}(w\mid x)=\frac{\widetilde{f}^{i}_{(W, X)}(w, x)}{\widetilde{f}^{i}_{X}(x)} =  \sum_{j} \frac{\sum_{l} \pi^{i}_{(W,Z,X),j,l,m} \mu_{Z}(S^{Z,i}_{l})}{\sum_{j,l} \pi^{i}_{(W,Z,X),j,l,m} \mu_{Z}(S^{Z,i}_{l}) \mu_{W}(S^{W,i}_{j})}  I\lbrace w\in S^{W,i}_{j}\rbrace. 
\end{align*}
Similarly, when $x\in S^{X,i}_{m}$,
\begin{align*}
    \widetilde{f}^{i}_{Y\mid X}(y\mid x)=\frac{\widetilde{f}^{i}_{(Y, X)}(y, x)}{\widetilde{f}^{i}_{X}(x)} =  \sum_{h} \frac{\sum_{l} \pi^{i}_{(Y,Z,X),h,l,m} \mu_{Z}(S^{Z,i}_{l})}{\sum_{j,l} \pi^{i}_{(W,Z,X),j,l,m} \mu_{Z}(S^{Z,i}_{l}) \mu_{W}(S^{W,i}_{j})}  I\lbrace y\in S^{Y,i}_{h}\rbrace. 
\end{align*}
Moreover, 
\begin{align*}
    \widetilde{f}^{i}_{(Z, X)}(z, x)
    &= \int \widetilde{f}^{i}_{(W,Z,X)}(w,z,x) d\mu_{W} 
    \\
    &= \sum_{l,m} \sum_{j} \pi^{i}_{(W,Z,X),j,l,m} \mu_{W}(S^{W,i}_{j})   I\lbrace z\in S^{Z,i}_{l}\rbrace I\lbrace x\in S^{X,i}_{m}\rbrace
\end{align*}
Now, \eqref{eq:L1conv_tostar_pre} implies that 
\begin{align*}
    &\widetilde{f}^{i}_{(Y,X)} \convlp f_{P^\ast,(Y,X)},
    \\
    &\widetilde{f}^{i}_{(W,X)} \convlp f_{P^\ast,(W,X)},
    \\
    &\widetilde{f}^{i}_{X} \convlp f_{P^\ast,X},
    \\
    &\widetilde{f}^{i}_{(Z,X)} \convlp f_{P^\ast,(Z,X)}.
\end{align*}

Recall that convergence in $L^{1}(\mu)$ implies convergence in measure under $\mu$, and that every sequence converging in measure has a subsequence converging almost everywhere.
Thus, passing to a subsequence if needed, we have that 
\begin{align*}
    &\widetilde{f}^{i}_{(Y,X)} \convas f_{P^\ast,(Y,X)},
    \\
    &\widetilde{f}^{i}_{(W,X)} \convas f_{P^\ast,(W,X)},
    \\
    &\widetilde{f}^{i}_{X} \convas f_{P^\ast,X},
    \\
    &\widetilde{f}^{i}_{(Z,X)} \convas f_{P^\ast,(Z,X)}.
\end{align*}
The last display, together with the assumption that $f_{P^\ast,X}$ is non-zero $\mu$-almost everywhere (see Condition \ref{cond:Past}), implies that
$$
\widetilde{f}^{i}_{W\mid X} = \frac{\widetilde{f}^{i}_{(W,X)}}{\widetilde{f}^{i}_{X}} \convas f_{P^\ast,W\mid X}
$$
and
$$
\widetilde{f}^{i}_{Y\mid X} = \frac{\widetilde{f}^{i}_{(Y,X)}}{\widetilde{f}^{i}_{X}} \convas f_{P^\ast,Y\mid X}.
$$
Since by assumption under $P^{\ast}$ it holds that $Y$ is independent of $(W,Z)$ given $X$, and $W$ is independent of $Z$ given $X$, we conclude that 
\begin{align}
    &\widetilde{f}^{i}_{Y\mid X} \widetilde{f}^{i}_{W\mid X} \widetilde{f}^{i}_{(Z,X)}\overset{\text{a.e. }\mu}{\to} f_{P^\ast},
    \label{eq:unifconv_tostar}
\end{align} 
which by Scheffe's lemma implies 
\begin{align}
    &\widetilde{f}^{i}_{Y\mid X} \widetilde{f}^{i}_{W\mid X} \widetilde{f}^{i}_{(Z,X)} \convlp f_{P^\ast}.
    \label{eq:L1conv_tostar}
\end{align} 
Let $\widetilde{P}_{i}$ be the law that has $\widetilde{f}^{i}_{Y\mid X} \widetilde{f}^{i}_{W\mid X} \widetilde{f}^{i}_{(Z,X)} $ as a density. Since $P^{\ast}\in\mathcal{R}$ and $\mathcal{R}$ satisfies Condition \ref{cond:R}, eventually passing to a subsequence if needed, we can assume that $\widetilde{P}_{i}\in\mathcal{R}$.

For future use, we define
\begin{align*}
    &\pi^{i}_{(W,X),j,m} = \sum_{l} \pi^{i}_{(W,Z,X),j,l,m} \mu_{Z}(S^{Z,i}_{l})
    \\
    &\pi^{i}_{X,m} = \sum_{j,l} \pi^{i}_{(W,Z,X),j,l,m} \mu_{Z}(S^{Z,i}_{l})\mu_{W}(S^{W,i}_{j})
    \\
    &\pi^{i}_{W\mid X,j,m} =\frac{\pi^{i}_{(W,X),j,m}}{\pi^{i}_{X,m}},
    \\
    &\pi^{i}_{Y\mid X,h,m} =\frac{\pi^{i}_{(Y,X),h,m}}{\pi^{i}_{X,m}},
    \\
    &\pi^{i}_{(Z,X),l,m} = \sum_{j} \pi^{i}_{(W,Z,X),j,l,m} \mu_{W}(S^{W,i}_{j}).
\end{align*}

% Since by assumption Condition \ref{cond:alpha_convergence_1} holds, equation \eqref{eq:unifconv_tostar} implies that, eventually passing to a subsequence if needed, we can assume that $\alpha_{\widetilde{P}_{i}}(W,X)$ converges to $\alpha_{P^{\ast}}(W,X)$ almost surely under $P^{\ast}$. Since $P^{\ast}$ satisfies
% \begin{equation}
%     var_{P^{\ast}}\left\lbrace \alpha_{P^{\ast}}(W,X)\mid X\right\rbrace > 0
%     \nonumber
% \end{equation}
% with positive probability under $P^{\ast}$,
% by Lemma \ref{lemma:positive_cond_var} we have that there exists $i_{0}$ such that for $i \geq i_{0}$
% \begin{equation}
%     var_{\widetilde{P}_{i}}\left\lbrace \alpha_{\widetilde{P}_{i}}(W,X)\mid X\right\rbrace > 0
%     \label{eq:lowerbound_var_alphatilde}
% \end{equation}
% with positive probability under $\widetilde{P}_{i}$.
% In what follows, we take $i\geq i_0$.

By Condition \ref{cond:alpha_coheres}, we have that $\widetilde{P}_{i}$ satisfies \eqref{eq:cont} and that $\alpha_{\widetilde{P}_{i}}(W,X)$ is equal $\widetilde{P}_{i}$-almost surely to a function of the form
\begin{equation} 
     \sum\limits_{j} \sum\limits_{m} \widetilde{\alpha}^{i}_{j,m} I\left\lbrace W \in S^{W,i}_{j} \right\rbrace I\left\lbrace X \in S^{X,i}_{m} \right\rbrace.
    \label{eq:alpha_tilde}
\end{equation}
Moreover, since by assumption Condition \ref{cond:alpha_convergence_1} holds, equation \eqref{eq:unifconv_tostar} implies that, eventually passing to a subsequence if needed, we can assume that $\alpha_{\widetilde{P}_{i}}(W,X)$ converges to $\alpha_{P^{\ast}}(W,X)$ almost surely under $P^{\ast}$. Since $P^{\ast}$ satisfies
\begin{equation}
    var_{P^{\ast}}\left\lbrace \alpha_{P^{\ast}}(W,X)\mid X\right\rbrace > 0
    \nonumber
\end{equation}
with positive probability under $P^{\ast}$,
by Lemma \ref{lemma:positive_cond_var} we have that there exists $i_{0}$ such that for $i \geq i_{0}$
\begin{equation}
    var_{\widetilde{P}_{i}}\left\lbrace \alpha_{\widetilde{P}_{i}}(W,X)\mid X\right\rbrace > 0
    \label{eq:lowerbound_var_alphatilde}
\end{equation}
with positive probability under $\widetilde{P}_{i}$.
In what follows, we take $i\geq i_0$.

\subsubsection*{Approximating $\widetilde{P}_{i}$ by a law $P_{i}$ under which $(Y,W)$ and $Z$ are not independent given $X$}

The marginal of $(Z,X)$ under $P_{i}$ will be given by $\widetilde{f}^{i}_{(Z,X)}$.
Next, we will construct the conditional distribution of $W$ given $(Z,X)$ under $P_{i}$.
Let 
\begin{align*}
    &\pi^{i}_{W\mid m}=(\pi^{i}_{W\mid X,1,m},\dots,\pi^{i}_{W\mid X,k^{i},m})^{\top}, \text{ for }m=1,\dots,k^{i}_{X},
\end{align*}
and let $\eta^{i}_{W}$ be a small (in absolute value) constant, which we will choose later. 
Let $\Pi^{W\mid Z,i,m} \in \mathbb{R}^{k^{i}\times k^{i}}$ be given by
\begin{equation}
    \Pi^{W\mid Z,i,m}= \widetilde{1} \left[\pi^{i}_{W\mid m}\right]^{\top} + \eta^{i}_{W} I,
\end{equation}
where $\widetilde{1}$ is the $k^{i}$-dimensional vector with all entries equal to 1, and $I\in \mathbb{R}^{k^{i}\times k^{i}}$ is the identity matrix. 
Let $\pi^{W\mid Z,i,m}_{l,j}$ be the entry in row $l$ and column $j$ of $\Pi^{W\mid Z,i,m}$.
For small enough absolute value of $\eta^{i}_{W}$, 
\begin{align}
    &\pi^{W\mid Z,i,m}_{l,j}>0  \text{ for all } l,j,
    \label{eq:piw_positive}
    \\
    & 1+\eta^{i}_{W}\mu_{W}(S^{W,i}_{l})>0 \text{ for all } l.
    \label{eq:normw_positive}
    \\
    & 1+(1/\eta^{i}_{W})\left[\pi^{i}_{W\mid m}\right]^{\top} \widetilde{1}\neq 0  \text{ for all } m.
    \label{eq:shermanw_positive}
\end{align}
The density of $W$ given $(Z,X)$ under $P_{i}$ will be given by 
\begin{equation}
f^{i}_{W\mid Z,X}(w\mid z,x):= \sum_{j,l,m} \frac{\pi^{W\mid Z,i,m}_{l,j}}{1+\eta^{i}_{W}\mu_{W}(S^{W,i}_{l})} I\left\lbrace z\in S^{Z,i}_{l}\right\rbrace I\left\lbrace w\in S^{W,i}_{j}\right\rbrace I\left\lbrace x\in S^{X,i}_{m}\right\rbrace. 
\label{eq:def_wzx}
\end{equation}
Note that $f^{i}_{W\mid Z,X}(w\mid z,x)$ is a well defined density. Indeed, $f^{i}_{W\mid Z,X}(w\mid z,x) \geq 0$ and for all $(z,x)\in S^{Z,i}_{l}\times S^{X,i}_{m}$
\begin{align*}
\int f^{i}_{W\mid Z,X}(w\mid z,x) d\mu_{W}&= \sum_{j}  \frac{\pi^{W\mid Z,i,m}_{l,j}}{1+\eta^{i}_{W}\mu_{W}(S^{W,i}_{l})} \mu_{W}(S^{W,i}_{j})
\\
&
=\sum_{j\neq l } \frac{\pi^{W\mid Z,i,m}_{l,j}}{1+\eta^{i}_{W}\mu_{W}(S^{W,i}_{l})} \mu_{W}(S^{W,i}_{j}) + \frac{\pi^{W\mid Z,i,m}_{l,l}}{1+\eta^{i}_{W}\mu_{W}(S^{W,i}_{l})} \mu_{W}(S^{W,i}_{l})
\\
&
= \sum_{j\neq l } \frac{\pi^{i}_{W\mid X,j,m}}{1+\eta^{i}_{W}\mu_{W}(S^{W,i}_{l})} \mu_{W}(S^{W,i}_{j}) + \frac{\pi^{i}_{W\mid X,l,m}+\eta_{W}^{i}}{1+\eta^{i}_{W}\mu_{W}(S^{W,i}_{l})} \mu_{W}(S^{W,i}_{l})
\\
&
= \sum_{j} \frac{\pi^{i}_{W\mid X,j,m}}{1+\eta^{i}_{W}\mu_{W}(S^{W,i}_{l})} \mu_{W}(S^{W,i}_{j}) + \frac{\eta_{W}^{i}}{1+\eta^{i}_{W}\mu_{W}(S^{W,i}_{l})} \mu_{W}(S^{W,i}_{l})
\\
&
= \frac{1}{1+\eta^{i}_{W}\mu_{W}(S^{W,i}_{l})}  + \frac{\eta_{W}^{i}\mu_{W}(S^{W,i}_{l})}{1+\eta^{i}_{W}\mu_{W}(S^{W,i}_{l})} 
\\
&=1,
\end{align*}
where in the next to last equality we used that
\begin{align*}
    1= \int \widetilde{f}^{i}_{W\mid X}(w\mid x) d\mu_{W} = \sum_{j} \pi^{i}_{W\mid X,j,m} \mu_{W}(S^{W,i}_{j}).
\end{align*}

Next, we will construct the conditional distribution of $Y$ given $(Z,X)$ under $P_{i}$. Let
\begin{align*}
    &\pi^{i}_{Y\mid m}=(\pi^{i}_{Y\mid X,1,m},\dots,\pi^{i}_{Y\mid X,k_{Y}^{i},m})^{\top}, \text{ for }m=1,\dots,k^{i}_{X}
\end{align*}
and let $\eta^{i}_{Y}$ be a small (in absolute value) constant, which we will choose later. 

Let $\widetilde{\alpha}_{m}^{i}$ be the vector with coordinates $\widetilde{\alpha}^{i}_{j,m}$ for $j=1,\dots,k^{i}$, $m=1,\dots,k_{X}^{i}$; see equation \eqref{eq:alpha_tilde}. Let $\iota^{i}_{y}$ be the vector with coordinates $\int_{S^{Y,i}_{h}} y \: d\mu_{Y}$ for $h=1,\dots,k_{Y}^{i}$ and $\widetilde{\mu}^{i}_{Y}$ be the vector with coordinates $\mu_{Y}(S^{Y,i}_{h})$, $h=1,\dots,k^{i}_{Y}$.
Recall that $\iota^{i}_{y}$ and  $\widetilde{\mu}^{i}_{Y}$ are not collinear. Then for $m=1,\dots,k^{i}_{X}$, we can find a matrix $M^{i,m}$ such that 
\begin{align*}
&M^{i,m}\iota^{i}_{y} = \widetilde{\alpha}_{m}^{i},
\\
&M^{i,m}\widetilde{\mu}^{i}_{Y} = 0.
\end{align*}
Let $\Pi^{Y\mid Z,i,m} \in \mathbb{R}^{k^{i}\times k_{Y}^{i}}$ be given by
\begin{equation}
    \Pi^{Y\mid Z,i,m}= \widetilde{1} \left[\pi^{i}_{Y\mid m}\right]^{\top} + \eta^{i}_{Y} M^{i,m}.
    \nonumber
\end{equation}
Why we have chosen to perturbate $\pi^{i}_{Y\mid m}$ using this matrix $M^{i,m}$ in particular will become aparent later on in the proof.
Let $\pi^{Y\mid Z,i,m}_{l,h}$ be the entry in row $l$ and column $h$ of $\Pi^{Y\mid Z,i,m}$.
For small enough $\eta^{i}_{Y}$, 
\begin{align}
    &\pi^{Y\mid Z,i,m}_{l,h}>0  \text{ for all } l,h,m,
    \label{eq:piy_positive}
    % \\
    % &1+\eta^{i}_{Y}\sum_{h} M^{i,m}_{l,h}\mu_{Y}(S^{Y,i}_{h}) >0 \text{ for all } l,m
    % \label{eq:normy_positive}
\end{align}
The density of $Y$ given $(Z,X)$ under $P_{i}$ will be given by
\begin{equation}
f^{i}_{Y\mid Z,X}(y\mid z,x):= \sum_{h,l} \pi^{Y\mid Z,i,m}_{l,h} I\left\lbrace z\in S^{Z,i}_{l}\right\rbrace I\left\lbrace y\in S^{Y,i}_{h}\right\rbrace I\left\lbrace x\in S^{X,i}_{m}\right\rbrace.
\label{eq:def_yzx}
\end{equation}
The proof that $f^{i}_{Y\mid Z}(y\mid z,x)$ is a density is analogous to the proof that $f^{i}_{W\mid Z}(w\mid z,x)$ is a density and is ommited here.
% which satisfies $f^{i}_{Y\mid Z}(y\mid z,x) \geq 0$ and for all $l,m$ and $z\in S^{Z,i}_{l}$, $x\in S^{X,i}_{m}$
% \[
% \int f^{i}_{Y\mid Z,X}(y\mid z,x) d\mu_{Y}= \sum_{h} \frac{\pi^{Y\mid Z,i,m}_{l,h}}{1+\eta^{i}_{Y}\sum_{h}M^{i,m}_{l,h}\mu_{Y}(S^{Y,i}_{h})} \mu_{Y}(S^{Y,i}_{h})=1. 
% \]

For any $h,l,m$, we have that 
\begin{align*}
    &\sup\limits_{(y,z,x)\in S^{Y,i}_{h} \times  S^{Z,i}_{l} \times S^{X,i}_{m}}
    \left\vert f^{i}_{Y\mid Z,X}(y\mid z,x) - \widetilde{f}^{i}_{Y\mid X}(y\mid x) \right\vert 
    = 
    \\ 
    &\sup\limits_{(y,z,x)\in S^{Y,i}_{h} \times  S^{Z,i}_{l} \times S^{X,i}_{m}}
  \left\vert \pi^{Y\mid Z,i,m}_{l,h} - \pi^{i}_{Y\mid m, h}\right\vert 
    =
    \\ 
    &\sup\limits_{(y,z,x)\in S^{Y,i}_{h} \times  S^{Z,i}_{l} \times S^{X,i}_{m}}
  \left\vert \pi^{i}_{Y\mid m, h}+\eta^{i}_{Y}M^{i,m}_{l,h} - \pi^{i}_{Y\mid m, h}\right\vert \underset{\eta^{i}_{Y}\to 0 }{\rightarrow} 0.
\end{align*}
Since the supports  $f^{i}_{Y\mid Z,X}$ and $\widetilde{f}^{i}_{Y\mid X}$ are $\cup_{h,l,m} \: S^{Y,i}_{h}\times S^{Z,i}_{l} \times S^{X,i}_{m}$ and $\cup_{h} S^{Y,i}_{h}\times S^{X,i}_{m}$ respectively, it follows that 
\begin{align}
    &f^{i}_{Y\mid Z,X} \convunif \widetilde{f}^{i}_{Y\mid X} \text{ when } \eta^{i}_{Y}\to 0.
    \label{eq:convunif_fy}
\end{align}
Arguing similarly, we get
\begin{align}
    f^{i}_{W\mid Z,X} \convunif \widetilde{f}^{i}_{W\mid X}\text{ when } \eta^{i}_{W}\to 0.
    \label{eq:convunif_fw}
\end{align}
Equations \eqref{eq:convunif_fy} and \eqref{eq:convunif_fw} imply that
\begin{equation}
f^{i}_{Y\mid Z,X}f^{i}_{W\mid Z,X} \widetilde{f}^{i}_{(Z,X)}  \overset{\text{a.e. }\mu}{\to} \widetilde{f}^{i}_{Y\mid X} \widetilde{f}^{i}_{W\mid X} \widetilde{f}^{i}_{(Z,X)} \text{ when } \eta^{i}_{W}\to 0, \eta^{i}_{Y}\to 0
\label{eq:unifconv_totilde}
\end{equation}
which by Scheffe's lemma implies 
\begin{equation}
    f^{i}_{Y\mid Z,X}f^{i}_{W\mid Z,X} \widetilde{f}^{i}_{(Z,X)} \convlp \widetilde{f}^{i}_{Y\mid X} \widetilde{f}^{i}_{W\mid X} \widetilde{f}^{i}_{(Z,X)} \text{ when } \eta^{i}_{W}\to 0, \eta^{i}_{Y}\to 0
\label{eq:L1conv_totilde}
\end{equation}
We will let $P_{i}$ be the law with density $f^{i}_{Y\mid Z,X}f^{i}_{W\mid Z,X} \widetilde{f}^{i}_{(Z,X)}$, where for the moment $\eta^{i}_{Y}$ and $\eta^{i}_{W}$ are arbitrary but small enough so that $f^{i}_{Y\mid Z,X}f^{i}_{W\mid Z,X} \widetilde{f}^{i}_{(Z,X)}$ is a density and \eqref{eq:piw_positive}, \eqref{eq:normw_positive} and \eqref{eq:shermanw_positive} hold. 
Note that
\begin{align*}
    &f_{P_{i},Y\mid Z,X} = f^{i}_{Y\mid Z,X},
    \\
    &f_{P_{i},W,\mid Z,X}=f^{i}_{W\mid Z,X},
    \\
    &f_{P_{i},(Z,X)}=\widetilde{f}^{i}_{(Z,X)}.
\end{align*}

Since $f_{P_{i}}$ is a rectangular non-negative simple function and Condition \ref{cond:alpha_coheres} holds by assumption, we have that $\alpha_{P_{i}}(W,X)$ is equal $P_{i}$-almost surely to a function of the form
\[ 
     \sum\limits_{j} \sum\limits_{m} \alpha^{i}_{j,m} I\left\lbrace W \in S^{W,i}_{j} \right\rbrace I\left\lbrace X \in S^{X,i}_{m} \right\rbrace.
\]

Since \eqref{eq:unifconv_totilde} holds,
Condition \ref{cond:alpha_convergence_1} implies that, eventually passing to a subsequence if needed, we can assume that 
$$\alpha_{P_{i}}(W,X) \to \alpha_{\widetilde{P}_{i}}(W,X)$$ when $\eta^{i}_{W}\to 0$, $\eta^{i}_{Y}\to 0$, almost surely under $\widetilde{P}_{i}$.
Let $\alpha_{m}^{i}$ be the vector with coordinates $\alpha^{i}_{j,m}$ for $j=1,\dots,k^{i}$, $m=1,\dots,k_{X}^{i}$. Lemma \ref{lemma:conv_alpha_params} then implies that
\begin{align}
    \alpha^{i}_{j,m} \to \widetilde{\alpha}^{i}_{j,m}  \text{ when } \eta^{i}_{W}\to 0, \eta^{i}_{Y}\to 0.
    \label{eq:convergence_alpha_totilde}
\end{align}

\subsubsection*{Showing that $P_{i}\in \mathcal{M}$}

 For any vector $v$, let $\diag(v)$ be the diagonal matrix with diagonal equal to $v$. Given two vectors $u,v$, we let $u/v$ stand for the vector that is formed by coordinate wise division of $u$ by $v$.

We have already argued that \eqref{eq:cont} holds when $P$ is replaced by $P_{i}$. Moreover, since $\widetilde{P}_{i}\in\mathcal{R}$, the set $\mathcal{R}$ satisfies Condition \ref{cond:R} and equation \eqref{eq:unifconv_totilde} holds, we can assume that $P_{i}\in\mathcal{R}$.
 Thus, to prove that $P_{i}\in\mathcal{M}$ we need to show that there exists a solution to \eqref{eq:integral_eq} under $P_{i}$ and that $\alpha_{P_{i}}\in \range(T^{\prime}_{P_{i}})$.

 \subsubsection*{Showing that there exists a solution to \eqref{eq:integral_eq} under $P_{i}$}
Note that for $g(W,X)$ to be a solution of \eqref{eq:integral_eq} under $P_{i}$, it suffices that 
\begin{equation}
    E_{P_{i}}\left\lbrace g(W,X)\mid Z,X\right\rbrace= E_{P_{i}}\left\lbrace Y\mid Z,X\right\rbrace \quad \text{whenever } (Z,X)\in S^{Z,i}_{l}\times S^{X,i}_{m} \text{ for any }l,m.
    \label{eq:integral_eq_discrete}
\end{equation}
We will show that there exists a solution to \eqref{eq:integral_eq_discrete} of the form 
\begin{equation}
    g(W,X) = \sum_{j}\sum_{m} \frac{g^{i}_{j,m}}{\mu_{W}(S^{W,i}_{j})} I\left\lbrace W\in S^{W,i}_{j}\right\rbrace I\left\lbrace X\in S^{X,i}_{m}\right\rbrace,
    \nonumber
\end{equation}
for some constants $g^{i}_{j,m}$. 
Now, when $(Z,X)\in S^{Z,i}_{l}\times S^{X,i}_{m}$, we have that
\begin{align*}
    E_{P_{i}}\left\lbrace Y\mid Z,X\right\rbrace = \sum_{h} \pi^{Y\mid Z,i,m}_{l,h} \int_{S^{Y,i}_{h}} y \: d\mu_{Y}
\end{align*}
and
\begin{align*}
    E_{P^{i}}\left\lbrace g(W,X)\mid Z,X\right\rbrace &=\sum_{j} \frac{\pi^{W\mid Z,i,m}_{l,j}}{1+\eta^{i}_{W}\mu_{W}(S^{W,i}_{l})} \int_{S^{W,i}_{j}} g(w,X) \: d\mu_{W}
    \\
    &=\sum_{j} \frac{\pi^{W\mid Z,i,m}_{l,j}}{1+\eta^{i}_{W}\mu_{W}(S^{W,i}_{l})} g^{i}_{j,m}
\end{align*}
where we used the definitions of $f^{i}_{Y\mid Z,X}$ and $f^{i}_{W\mid Z,X}$ given in \eqref{eq:def_yzx} and \eqref{eq:def_wzx} respectively.

Let $g^{i}_{m}$ be the vector with coordinates $g^{i}_{j,m}$, $j=1,\dots,k^{i}$.
 Let $\widetilde{\mu}^{i}_{W}$ be the vector with coordinates $\mu_{W}(S^{W,i}_{l})$ for $l=1,\dots,k^{i}$.
Then $g$ satisfying \eqref{eq:integral_eq_discrete} is equivalent to the following linear system holding for each $m$
\begin{equation}
\left\lbrace  \diag(1+ \eta^{i}_{W} \widetilde{\mu}^{i}_{W})\right\rbrace^{-1} \Pi^{W\mid Z,i,m} g^{i}_{m} = \Pi^{Y\mid Z,i,m} \iota^{i}_{y}.
\label{eq:main_linear_system}
\end{equation}
In what follows, to make the notation lighter we will ommit all $i$ indices.

Since by \eqref{eq:shermanw_positive} it holds that 
$$
1+(1/\eta_{W})\left[\pi_{W\mid m}\right]^{\top} \widetilde{1}\neq 0
$$ 
for all $m$, by the Sherman-Morrison formula $\Pi^{W\mid Z,m}$ is invertible for all $m$ and
\[
    \left\lbrace \Pi^{W\mid Z,m}\right\rbrace^{-1}=\frac{1}{\eta_{W}}I - \frac{1}{\eta_{W}\pi_{W\mid m}^{\top} \widetilde{1}+\eta_{W}^{2}} \widetilde{1} \pi_{W\mid m}^{\top}.
\]
It follows that if 
\begin{equation}
g^{\dagger}_{m}:=  \left\lbrace \frac{1}{\eta_{W}}I - \frac{1}{\eta_{W}\pi_{W\mid m}^{\top} \widetilde{1}+\eta_{W}^{2}} \widetilde{1} \pi_{W\mid m}^{\top} \right\rbrace \diag\left(1+ \eta_{W} \widetilde{\mu}_{W}\right) \Pi^{Y\mid Z,m} \iota_{y} 
\label{eq:integral_g_discrete}
\end{equation}
then $g^{\dagger}(W,X)$ defined by 
\begin{equation}
    g^{\dagger}(W,X) = \sum_{j}\sum_{m} \frac{g^{^{\dagger}}_{j,m}}{\mu_{W}(S^{W}_{j})} I\left\lbrace W\in S^{W}_{j}\right\rbrace I\left\lbrace X\in S^{X}_{m}\right\rbrace,
    \label{eq:def_g_discrete_solution}
\end{equation}
is a solution to \eqref{eq:integral_eq_discrete}. 

\subsubsection*{Showing that $\alpha_{P_{i}}(W,X)\in \range(T^{\prime}_{P_{i}})$}
We need to show that there exists $h(Z,X)\in L^{2}\lbrace P_{(Z,X)}\rbrace$ such that
\begin{equation}
    E_{P}\left\lbrace h(Z,X)\mid W,X\right\rbrace = \alpha_{P}(W,X)
    \label{eq:riesz_in_range}
\end{equation}
almost surely under $P$. It suffices to show that \eqref{eq:riesz_in_range} holds whenever $(W,X)\in S^{W}_{j}\times S^{X}_{m}$ for all $j,m$. Note that
\begin{align*}
    f_{P,Z\mid W,X} = \frac{f_{P,W\mid Z,X}\widetilde{f}_{(Z,X)}}{f_{P,(W,X)}}.
\end{align*}
It follows that when $(w,x)\in S^{W}_{j}\times S^{X}_{m}$, 
\begin{align*}
    f_{P,Z\mid W,X}(z\mid w,x) = c_{j,m} f_{P,W\mid Z,X}(w\mid z,x) \widetilde{f}_{(Z,X)}(z,x)
\end{align*}
for a positive constant $c_{j,m}$. Now, consider $h(Z,X)$ defined by
\begin{equation*}
    h(Z,X) = \sum\limits_{l,m} \frac{h_{l,m}}{\mu_{Z}(S^{Z}_{l})} I\left\lbrace Z \in S^{Z}_{l}\right\rbrace I\left\lbrace X \in S^{X}_{m}\right\rbrace,
\end{equation*}
for some constants $h_{l,m}$. Then, when $(W,X)\in S^{W}_{j}\times S^{X}_{m}$
\begin{align*}
    E_{P}\left\lbrace h(Z,X)\mid W,X\right\rbrace &= c_{j,m} \int h(z,X) f_{P,Z\mid W,X}(z\mid W,X) \widetilde{f}_{(Z,X)}(z,X)d\mu_{Z}  
    \\
    &=c_{j,m} \sum_{l} \frac{\pi^{W\mid Z,m}_{l,j}}{1+\eta_{W}\mu_{W}(S^{W}_{l})} \pi_{(Z,X),l,m}  h_{l,m}.
\end{align*}
Let $h_{m}$ be the vector with coordinates $h_{l,m}$ for $l=1,\dots,k$, let $c_{m}$ be the vector with coordinates $c_{j,m}$ for $j=1,\dots,k$ and let $\pi_{(Z,X),m}$ be the vector with coordinates $\pi_{(Z,X),l,m}$ for $l=1,\dots,k$.
It follows that \eqref{eq:riesz_in_range} holds for $h(Z,X)$ if and only if for all $m$
\begin{align*}
    h^{\top}_{m} \left\lbrace\diag(1+ \eta_{W} \widetilde{\mu}_{W})\right\rbrace^{-1} \diag(\pi_{(Z,X),m})\Pi^{W\mid Z,m} \diag(c_{m}) = \alpha_{m}
\end{align*}
Since all matrices on the left hand side of the last display are invertible, for each $m$ there exists a solution $h^{\dagger}_{m}$. It follows that 
\begin{equation*}
    h^{\dagger}(Z,X) = \sum\limits_{l,m} \frac{h^{\dagger}_{l,m}}{\mu_{Z}(S^{Z}_{l})} I\left\lbrace Z \in S^{Z}_{l}\right\rbrace I\left\lbrace X \in S^{X}_{m}\right\rbrace,
\end{equation*}
satisfies \eqref{eq:riesz_in_range}. Thus, we have shown that $P\in\mathcal{M}$.

% Next, we need to show that there exists $g(W,X)\in L^{2}\lbrace P_{(W,X)}\rbrace$ such that
% \begin{equation}
%     E_{P}\left\lbrace g(W,X)\mid Z,X\right\rbrace = h^{\dagger}(Z,X)
%     \label{eq:riesz_in_range_2}
% \end{equation}
% almost surely under $P$. It suffices to show that \eqref{eq:riesz_in_range_2} holds whenever $(Z,X)\in S^{Z}_{l}\times S^{X}_{m}$ for all $l,m$. 
% Consider a function of the form 
% \begin{equation*}
%     g(W,X) = \sum\limits_{j,m} \frac{g_{j,m}}{\mu_{W}(S^{W}_{j})} I\left\lbrace W \in S^{W}_{j}\right\rbrace I\left\lbrace X \in S^{X}_{m}\right\rbrace,
% \end{equation*}
% for some constants $g_{j,m}$.
% Then \eqref{eq:riesz_in_range_2} holding for $g$ is equivalent to the following equation holding for all $m$
% \begin{equation}
%     \left\lbrace  \diag(1+ \eta_{W} \widetilde{\mu}_{W})\right\rbrace^{-1} \Pi^{W\mid Z,m} g_{m} = h^{\dagger}_{m},
%     \nonumber
% \end{equation}
% where $g_{m}$ is the vector with coordinates $g_{j,m}$, $j=1,\dots,k$. Since all the matrices on the left hand side of the last display are invertible, a solution $g^{\ast}_{m}$ exists for all $m$. It is easy to check that 
% \begin{equation*}
%     g^{\ast}(W,X) = \sum\limits_{j,m} \frac{g^{\ast}_{j,m}}{\mu_{W}(S^{W}_{j})} I\left\lbrace W \in S^{W}_{j}\right\rbrace I\left\lbrace X \in S^{X}_{m}\right\rbrace,
% \end{equation*}
% is a solution to \eqref{eq:riesz_in_range_2}. Thus, we shown that $\alpha_{P}(W,X)\in \range(T^{\prime}_{P}T_{P})$.

\subsubsection*{Obtaining a closed form expression for $\varphi(P_{i})$}
Note that
\begin{align*}
f_{P,(W,X)}(w,x) &= \int f_{W\mid Z,X}(w\mid z,x) \widetilde{f}_{(Z,X)}(z,x) d\mu_{Z} 
\\
&=\sum_{j,l,m} \frac{\pi^{W\mid Z,m}_{l,j}}{1+\eta^{i}_{W}\mu_{W}(S^{W}_{l})} \mu_{Z}(S^{Z}_{l}) \pi_{(Z,X),l,m} I\left\lbrace w\in S^{W}_{j}\right\rbrace I\left\lbrace x\in S^{X}_{m}\right\rbrace . 
\end{align*}
Then, for $g^{\dagger}(W,X)$ defined in \eqref{eq:def_g_discrete_solution}
\begin{align*}
    &E_{P}\left\lbrace g^{\dagger}(W,X) \alpha_{P}(W,X)\right\rbrace 
    =
    \\
    &\sum_{j,l,m} \frac{\pi^{W\mid Z,m}_{l,j}}{1+\eta^{i}_{W}\mu_{W}(S^{W}_{l})} \mu_{Z}(S^{Z}_{l}) \pi_{(Z,X),l,m}  \int_{S^{X}_{m}}\int_{S^{W}_{j}} g^{\dagger}(w,x) \alpha_{P}(w,x) d\mu_{W}d\mu_{X}
    =
    \\
    &
    \sum_{j,l,m} \frac{\pi^{W\mid Z,m}_{l,j}}{1+\eta_{W}\mu_{W}(S^{W}_{l})} \mu_{Z}(S^{Z}_{l}) \pi_{(Z,X),l,m} \alpha_{j,m} \int_{S^{X}_{m}}\int_{S^{W}_{j}} g^{\dagger}(w,x)  d\mu_{W}d\mu_{X}
    \\
    &
    \sum_{j,l,m} \frac{\pi^{W\mid Z,m}_{l,j}}{1+\eta_{W}\mu_{W}(S^{W}_{l})} \mu_{Z}(S^{Z}_{l}) \pi_{(Z,X),l,m}  \alpha_{j,m} \mu_{X}(S^{X}_{m}) g^{\dagger}_{j,m}.
\end{align*}
It follows from the previous display that
\begin{equation}
    \varphi(P)= \sum_{m} \mu_{X}(S^{X}_{m}) \left\lbrace \diag\left\lbrace\pi_{(Z,X),m}\right\rbrace \widetilde{\mu}_{Z}\right\rbrace^{\top} \left\lbrace \diag(1+ \eta_{W} \widetilde{\mu}_{W})\right\rbrace^{-1} \Pi^{W\mid Z,m}\diag(\alpha_{m})g^{\dagger}_{m},
    \label{eq:varphi_discrete_decomposition}
\end{equation}
where $\widetilde{\mu}_{Z}$ is the vector with coordinates $\mu_{Z}(S^{Z}_{l})$, $l=1,\dots,k$.
%  and $\pi_{(Z,X),m}$ is the vector with coordinates $\pi_{(Z,X),l,m}$ for $l=1,\dots,k$. 

Now, putting together \eqref{eq:integral_g_discrete} and \eqref{eq:varphi_discrete_decomposition} we have that $\varphi(P)$
is equal to 
\begin{align}
    &\sum_{m}  \mu_{X}(S^{X}_{m}) \left\lbrace \diag\left\lbrace\pi_{(Z,X),m}\right\rbrace \widetilde{\mu}_{Z}\right\rbrace^{\top} \left\lbrace \diag(1+ \eta_{W} \widetilde{\mu}_{W})\right\rbrace^{-1} \Pi^{W\mid Z}\diag(\alpha_{m}) \left\lbrace \frac{1}{\eta_{W}}I - \frac{1}{\eta_{W}\pi_{W\mid m}^{\top}\widetilde{1}+\eta_{W}^{2}} \widetilde{1} \pi_{W\mid m}^{\top} \right\rbrace 
    \nonumber
    \\
    &
    \diag\left(\widetilde{1}+ \eta_{W} \widetilde{\mu}_{W}\right) \Pi^{Y\mid Z,m} \iota_{y}.
    \label{eq:main_expression_varphi}
\end{align}

\subsubsection*{Choosing the perturbation factors so that $\varphi(P_{i})=\zeta$}

Now, suppose we take $\eta_{Y}=\gamma \eta_{W}$, for some $\gamma \in\mathbb{R}$. Let $\psi(\gamma,\eta_{W})=\varphi(P_{i})$, taking $P_{i}$ implicitly as a function of $\gamma$ and $\eta_{W}$. Then $\psi(\gamma,\eta_{W})$ is a continuous function over $\mathbb{R}\times \left(\mathbb{R}\setminus\lbrace 0 \rbrace\right)$. We will show that there exists an increasing function $\zeta\mapsto \gamma(\zeta)$ such that $\psi(\gamma(\zeta),\eta_{W}) \to \zeta$ when $\eta_{W} \to 0$ for all $\zeta\in\mathbb{R}$. Lemma \ref{lemma:intermediate_value} would then imply that there exists a sequence $\lbrace (\gamma_{t},\eta_{W,t})\rbrace_{t\geq 1}$ such that $\eta_{W,t}\to 0$ as $t \to \infty$, $(\gamma_{t})_{t\geq 1}$ is bounded and $\psi(\gamma_{t},\eta_{W,t})=\zeta$ for all $t$. Recalling next that \eqref{eq:L1conv_tostar} and \eqref{eq:L1conv_totilde} hold, choosing a $t_{i}$ such that $\eta_{W,t_{i}}$ is small enough so that \eqref{eq:piw_positive}, \eqref{eq:normw_positive}, \eqref{eq:shermanw_positive}, \eqref{eq:piy_positive} hold and moreover 
$$
\Vert f^{i}_{Y\mid Z,X}f^{i}_{W\mid Z,X} \widetilde{f}^{i}_{(Z,X)} - \widetilde{f}^{i}_{Y\mid X} \widetilde{f}^{i}_{W\mid X} \widetilde{f}^{i}_{(Z,X)} \Vert_{L^{1}(\mu)} < 1/i,
$$
we would then arrive at the conclusion that $P_{i}$ satisfies all the requirements stated in Theorem \ref{theo:main}, we would have then arrived at the proof of the theorem. 

It remains to show that  there exists an increasing function $\zeta\mapsto \gamma(\zeta)$ such that $\psi(\gamma(\zeta),\eta_{W}) \to \zeta$ when $\eta_{W} \to 0$ for all $\zeta\in\mathbb{R}$. To do so, we start by further expanding the expression inside the summation in \eqref{eq:main_expression_varphi}. We will expand it from right to left. For now, we will ignore the $\mu_{X}(S^{X}_{m})$ factor. Fix a value of $m$. The first product $ \Pi^{Y\mid Z,m} \iota_{y}$ is
\[
    \widetilde{1} \pi_{Y\mid m}^{\top}\iota_{y} + \eta_{Y}M^{m} \iota_{y}.
\]
Multiplying this product by $\diag\left(\widetilde{1}+ \eta_{W} \widetilde{\mu}_{W}\right)$ we get 
\begin{align*}
    &\diag\left(\widetilde{1}+ \eta_{W} \widetilde{\mu}_{W}\right) \widetilde{1} \pi_{Y\mid m}^{\top}\iota_{y} + \eta_{Y} \diag\left(\widetilde{1}+ \eta_{W} \widetilde{\mu}_{W}\right) M^{m} \iota_{y}=
    \\
    &\left(\widetilde{1}+ \eta_{W} \widetilde{\mu}_{W}\right)  \pi_{Y\mid m}^{\top}\iota_{y} + \eta_{Y} \diag\left(\widetilde{1}+ \eta_{W} \widetilde{\mu}_{W}\right) M^{m}\iota_{y}.
\end{align*}
Next, multiplying the right hand side in the last display by $\left\lbrace \frac{1}{\eta_{W}}I - \frac{1}{\eta_{W}\pi_{W\mid m}^{\top}\widetilde{1}+\eta_{W}^{2}} \widetilde{1} \pi_{W\mid m}^{\top} \right\rbrace $ yields
\begin{align*}
    &\frac{1}{\eta_{W}}\left(\widetilde{1}+ \eta_{W} \widetilde{\mu}_{W}\right)  \pi_{Y\mid m}^{\top}\iota_{y} + \frac{ \eta_{Y}}{\eta_{W}} \diag\left(\widetilde{1}+ \eta_{W} \widetilde{\mu}_{W}\right) M^{m}\iota_{y} 
    \\
    &
    - \frac{1}{\eta_{W}\pi_{W\mid m}^{\top} \widetilde{1}+\eta_{W}^{2}} \widetilde{1} \pi_{W\mid m}^{\top} \left(\widetilde{1}+ \eta_{W} \widetilde{\mu}_{W}\right)  \pi_{Y\mid m}^{\top}\iota_{y}
    -  \frac{\eta_{Y}}{\eta_{W}\pi_{W\mid m}^{\top} \widetilde{1}+\eta_{W}^{2}} \widetilde{1} \pi_{W\mid m}^{\top}  \diag\left(\widetilde{1}+ \eta_{W} \widetilde{\mu}_{W}\right) M^{m}\iota_{y}
\end{align*}
Multiplying the expression in the last display by $\diag(\alpha_{m})$ we get
\begin{align*}
    &\frac{1}{\eta_{W}} \diag(\alpha_{m}) \left(\widetilde{1}+ \eta_{W} \widetilde{\mu}_{W}\right)  \pi_{Y\mid m}^{\top}\iota_{y} + \frac{ \eta_{Y}}{\eta_{W}} \diag(\alpha_{m}) \diag\left(\widetilde{1}+ \eta_{W} \widetilde{\mu}_{W}\right) M^{m}\iota_{y} 
    \\
    &
    - \frac{1}{\eta_{W}\pi_{W\mid m}^{\top} \widetilde{1}+\eta_{W}^{2}} \alpha_{m} \pi_{W\mid m}^{\top} \left(\widetilde{1}+ \eta_{W} \widetilde{\mu}_{W}\right)  \pi_{Y\mid m}^{\top}\iota_{y}
    -  \frac{\eta_{Y}}{\eta_{W}\pi_{W\mid m}^{\top} \widetilde{1}+\eta_{W}^{2}} \alpha_{m}\pi_{W\mid m}^{\top}  \diag\left(\widetilde{1}+ \eta_{W} \widetilde{\mu}_{W}\right) M^{m}\iota_{y}.
\end{align*}
Multiplying the expression above by 
\[
    \widetilde{1} \pi_{W\mid m}^{\top} + \eta_{W}I,
\]
gives 
\begin{align}
    &\frac{1}{\eta_{W}}  \widetilde{1} \pi_{W\mid m}^{\top} \diag(\alpha_{m}) \left(\widetilde{1}+ \eta_{W} \widetilde{\mu}_{W}\right)  \pi_{Y\mid m}^{\top}\iota_{y} + \frac{ \eta_{Y}}{\eta_{W}}  \widetilde{1} \pi_{W\mid m}^{\top} \diag(\alpha_{m}) \diag\left(\widetilde{1}+ \eta_{W} \widetilde{\mu}_{W}\right) M^{m}\iota_{y} + 
    \label{eq:decomposition_eight_terms}
    \\
    &\diag(\alpha_{m}) \left(\widetilde{1}+ \eta_{W} \widetilde{\mu}_{W}\right)  \pi_{Y\mid m}^{\top}\iota_{y} + 
    \eta_{Y}  \diag(\alpha_{m}) \diag\left(\widetilde{1}+ \eta_{W} \widetilde{\mu}_{W}\right) M^{m}\iota_{y} -
    \nonumber
    \\
    & \frac{1}{\eta_{W}\pi_{W\mid m}^{\top} \widetilde{1}+\eta_{W}^{2}}  \widetilde{1} \pi_{W\mid m}^{\top} \alpha_{m} \pi_{W\mid m}^{\top} \left(\widetilde{1}+ \eta_{W} \widetilde{\mu}_{W}\right)  \pi_{Y\mid m}^{\top}\iota_{y} -
    \nonumber
    \\
    &
    \frac{\eta_{Y}}{\eta_{W}\pi_{W\mid m}^{\top} \widetilde{1}+\eta_{W}^{2}} \widetilde{1} \pi_{W\mid m}^{\top} \alpha_{m}\pi_{W\mid m}^{\top}  \diag\left(\widetilde{1}+ \eta_{W} \widetilde{\mu}_{W}\right) M^{m}\iota_{y}-
    \nonumber
    \\
    & - \frac{1}{\pi_{W\mid m}^{\top}\widetilde{1}+\eta_{W}} \alpha_{m} \pi_{W\mid m}^{\top} \left(\widetilde{1}+ \eta_{W} \widetilde{\mu}_{W}\right)  \pi_{Y\mid m}^{\top}\iota_{y}
    -  \frac{\eta_{Y}}{\pi_{W\mid m}^{\top}\widetilde{1}+\eta_{W}} \alpha_{m} \pi_{W\mid m}^{\top}  \diag\left(\widetilde{1}+ \eta_{W} \widetilde{\mu}_{W}\right) M^{m}\iota_{y}.
    \nonumber
\end{align}
Finally, we need to left multiply each of the eight terms in \eqref{eq:decomposition_eight_terms} by
\[
    \left\lbrace \diag\left\lbrace\pi_{(Z,X),m}\right\rbrace \widetilde{\mu}_{Z}\right\rbrace^{\top} \left\lbrace \diag(1+ \eta_{W} \widetilde{\mu}_{W})\right\rbrace^{-1}.
\]

% Recall that we need to show that for any $\zeta$ we can find $\gamma$ such that when we take $\eta_{Y}=\gamma \eta_{W}$, it holds that $\varphi(P)\to \zeta$ when $\eta_{W}\to 0$. 
% We will analyse the terms in the big display in \eqref{eq:decomposition_eight_terms}. 
Recall that, in paragraph leading to \eqref{eq:convergence_alpha_totilde}, we showed that for all $j,m$
\begin{align}
    \alpha_{j,m} \to \widetilde{\alpha}_{j,m}  \text{ when } \eta_{W}\to 0.
    \nonumber 
\end{align}

The fourth and eight term the summation in display \eqref{eq:decomposition_eight_terms} converge to zero when $\eta_{W}\to 0$. 

On the other hand, the sum of the first and fifth term is equal to
\begin{align}
    &\frac{1}{\eta_{W}}  \widetilde{1} \pi_{W\mid m}^{\top} \diag(\alpha_{m}) \left(\widetilde{1}+ \eta_{W} \widetilde{\mu}_{W}\right)  \pi_{Y\mid m}^{\top}\iota_{y} -  \frac{1}{\eta_{W}\pi_{W\mid m}^{\top} \widetilde{1}+\eta_{W}^{2}}  \widetilde{1} \pi_{W\mid m}^{\top} \alpha_{m} \pi_{W\mid m}^{\top} \left(\widetilde{1}+ \eta_{W} \widetilde{\mu}_{W}\right)  \pi_{Y\mid m}^{\top} \iota_{y}=
    % &\frac{1}{\eta_{W}} \widetilde{1} \pi_{W\mid m}^{\top} \left[ \diag(\alpha) \left(\frac{1+ \eta_{W} \widetilde{\mu}_{W}}{1+ \eta_{Y} \widetilde{N}}\right) - \frac{1}{\pi_{W\mid m}^{\top} \widetilde{1}+\eta_{W}} \alpha \pi_{W\mid m}^{\top} \left(\frac{1+ \eta_{W} \widetilde{\mu}_{W}}{1+ \eta_{Y} \widetilde{N}}\right) \right] \pi_{Y\mid m}^{\top}\iota_{y} = 
    \nonumber
    \\
    &\frac{1}{\eta_{W}} \widetilde{1} \pi_{W\mid m}^{\top} \left[ \left\lbrace \diag(\alpha_{m})  - \frac{1}{\pi_{W\mid m}^{\top} \widetilde{1}+\eta_{W}} \alpha_{m} \pi_{W\mid m}^{\top} \right\rbrace \left(\widetilde{1}+ \eta_{W} \widetilde{\mu}_{W}\right) \right]\pi_{Y\mid m}^{\top}\iota_{y}.
    \label{eq:decomp_infty_term}
\end{align}
We can decompose the term in brackets in \eqref{eq:decomp_infty_term} as follows
\begin{align*}
    &\left\lbrace \diag(\alpha_{m})  - \frac{1}{\pi_{W\mid m}^{\top} \widetilde{1}+\eta_{W}} \alpha_{m} \pi_{W\mid m}^{\top} \right\rbrace \widetilde{1} 
    +
    \left\lbrace \diag(\alpha_{m})  - \frac{1}{\pi_{W\mid m}^{\top} \widetilde{1}+\eta_{W}} \alpha_{m} \pi_{W\mid m}^{\top}   \right\rbrace \left\lbrace \left(\widetilde{1}+ \eta_{W} \widetilde{\mu}_{W}\right) - \widetilde{1}\right\rbrace=
    \\
    & \left\lbrace \alpha_{m}  - \frac{1}{\pi_{W\mid m}^{\top} \widetilde{1}+\eta_{W}} \alpha_{m} \pi_{W\mid m}^{\top}\widetilde{1} \right\rbrace 
    + 
    \left\lbrace \diag(\alpha_{m})  - \frac{1}{\pi_{W\mid m}^{\top} \widetilde{1}+\eta_{W}} \alpha_{m} \pi_{W\mid m}^{\top}   \right\rbrace   \eta_{W} \widetilde{\mu}_{W}=
    \\
    & \left\lbrace \frac{\alpha_{m}\pi_{W\mid m}^{\top} \widetilde{1}+\eta_{W}\alpha_{m}-\alpha_{m} \pi_{W\mid m}^{\top}\widetilde{1}}{\pi_{W\mid m}^{\top} \widetilde{1}+\eta_{W}} \right\rbrace 
    + 
    \left\lbrace \diag(\alpha_{m})  - \frac{1}{\pi_{W\mid m}^{\top} \widetilde{1}+\eta_{W}} \alpha_{m} \pi_{W\mid m}^{\top}   \right\rbrace   \eta_{W} \widetilde{\mu}_{W}=
    \\
    & \left\lbrace \frac{\eta_{W}\alpha_{m}}{\pi_{W\mid m}^{\top} \widetilde{1}+\eta_{W}} \right\rbrace 
    + 
    \left\lbrace \diag(\alpha_{m})  - \frac{1}{\pi_{W\mid m}^{\top} \widetilde{1}+\eta_{W}} \alpha_{m} \pi_{W\mid m}^{\top}   \right\rbrace  \eta_{W} \widetilde{\mu}_{W}
\end{align*}
Thus, when $\eta_{W}\to 0$, \eqref{eq:decomp_infty_term} converges to
\begin{align*}
    &\frac{ \widetilde{1} \pi_{W\mid m}^{\top}\widetilde{\alpha}_{m}\pi_{Y\mid m}^{\top}\iota_{y}}{\pi_{W\mid m}^{\top} \widetilde{1}} + 
    \widetilde{1} \pi_{W\mid m}^{\top}
    \left\lbrace  \frac{ \diag(\widetilde{\alpha}_{m})\pi_{W\mid m}^{\top} \widetilde{1}  - \widetilde{\alpha}_{m} \pi_{W\mid m}^{\top}}{\pi_{W\mid m}^{\top} \widetilde{1}}    \right\rbrace  \widetilde{\mu}_{W} \pi_{Y\mid m}^{\top}\iota_{y}=
    \\
    & \frac{\widetilde{1} \pi_{W\mid m}^{\top}}{\pi_{W\mid m}^{\top} \widetilde{1}} \left[ \widetilde{\alpha}_{m} + 
    \left\lbrace  \diag(\widetilde{\alpha}_{m})\pi_{W\mid m}^{\top} \widetilde{1}  - \widetilde{\alpha}_{m} \pi_{W\mid m}^{\top}    \right\rbrace  \widetilde{\mu}_{W}  \right]\pi_{Y\mid m}^{\top}\iota_{y}.
    % =
    % \\
    % & \frac{\widetilde{1} \pi_{W\mid m}^{\top}}{\pi_{W\mid m}^{\top} \widetilde{1}} \left[ \widetilde{\alpha}_{m} + 
    % \left\lbrace  \diag(\widetilde{\alpha}_{m})\pi_{W\mid m}^{\top} \widetilde{1}  - \widetilde{\alpha}_{m} \pi_{W\mid m}^{\top}    \right\rbrace  \widetilde{\mu}_{W}  \right]\pi_{Y\mid m}^{\top}\iota_{y}.
\end{align*}

For the sum of the third and seventh terms in \eqref{eq:decomposition_eight_terms} we have that
\begin{align*}
    &\diag(\alpha_{m}) \left(\widetilde{1}+ \eta_{W} \widetilde{\mu}_{W}\right)  \pi_{Y\mid m}^{\top}\iota_{y} -  \frac{1}{\pi_{W\mid m}^{\top}\widetilde{1}+\eta_{W}} \alpha_{m} \pi_{W\mid m}^{\top} \left(\widetilde{1}+ \eta_{W} \widetilde{\mu}_{W}\right)  \pi_{Y\mid m}^{\top}\iota_{y} \underset{\eta_{W}\to 0 }{\rightarrow}
    &
    \\
    &\widetilde{\alpha}_{m} \pi_{Y\mid m}^{\top}\iota_{y} - \frac{1}{\pi_{W\mid m}^{\top}\widetilde{1}} \widetilde{\alpha}_{m} \pi_{W\mid m}^{\top} \widetilde{1} \pi_{Y\mid m}^{\top}\iota_{y} = 0.
\end{align*}

Finally, for the sum of the second and sixth terms  in \eqref{eq:decomposition_eight_terms} we have that 
\begin{align*}
    &\frac{ \eta_{Y}}{\eta_{W}}  \widetilde{1} \pi_{W\mid m}^{\top} \diag(\alpha_{m}) \diag\left(\widetilde{1}+ \eta_{W} \widetilde{\mu}_{W}\right) M^{m}\iota_{y} - \frac{\eta_{Y}}{\eta_{W}\pi_{W\mid m}^{\top} \widetilde{1}+\eta_{W}^{2}} \widetilde{1} \pi_{W\mid m}^{\top} \alpha_{m} \pi_{W\mid m}^{\top}  \diag\left(\widetilde{1}+ \eta_{W} \widetilde{\mu}_{W}\right) M^{m}\iota_{y}=
    \\
    &
    \frac{\eta_{Y}}{\eta_{W}}  \widetilde{1} \pi_{W\mid m}^{\top} 
    \left[ 
        \diag(\alpha_{m})  - 
         \frac{1}{\pi_{W\mid m}^{\top}\widetilde{1}+\eta_{W}}  \alpha_{m}\pi_{W\mid m}^{\top}  
    \right] \diag\left(1+ \eta_{W} \widetilde{\mu}_{W}\right) M^{m}\iota_{y} \underset{\eta_{W}\to 0 }{\rightarrow}
    \\
    &
    \gamma  \widetilde{1} \pi_{W\mid m}^{\top} 
    \left[ 
        \diag(\widetilde{\alpha}_{m})  - 
         \frac{\widetilde{\alpha}_{m}\pi_{W\mid m}^{\top}  }{\pi_{W\mid m}^{\top}\widetilde{1}}  
    \right] M^{m}\iota_{y} = 
    \\
    &
    \gamma  \widetilde{1} 
    \left[ 
        \pi_{W\mid m}^{\top} \diag(\widetilde{\alpha}_{m})  - 
         \frac{\pi_{W\mid m}^{\top} \widetilde{\alpha}_{m}\pi_{W\mid m}^{\top}  }{\pi_{W\mid m}^{\top}\widetilde{1}}  
    \right] M^{m}\iota_{y} =
    \\
    &
    \gamma  \widetilde{1} 
    \left[ 
        \pi_{W\mid m}^{\top} \diag(\widetilde{\alpha}_{m}) M^{m}\iota_{y}   - 
        \frac{\pi_{W\mid m}^{\top} \widetilde{\alpha}_{m}\pi_{W\mid m}^{\top}  }{\pi_{W\mid m}^{\top}\widetilde{1}}     M^{m}\iota_{y} 
    \right]
\end{align*}
In conclusion, as $\eta_{W} \to 0$, the sum in display\eqref{eq:decomposition_eight_terms} converges to
\begin{align}
   &\gamma  \widetilde{1} 
    \left[ 
        \pi_{W\mid m}^{\top} \diag(\widetilde{\alpha}_{m}) M^{m}\iota_{y}   - \frac{\pi_{W\mid m}^{\top} \widetilde{\alpha}_{m}\pi_{W\mid m}^{\top}  }{\pi_{W\mid m}^{\top}\widetilde{1}}     M^{m}\iota_{y}  
    \right] + 
    \nonumber
    \\
    &\frac{\widetilde{1} \pi_{W\mid m}^{\top}}{\pi_{W\mid m}^{\top} \widetilde{1}} \left[ \widetilde{\alpha}_{m} + 
    \left\lbrace  \diag(\widetilde{\alpha}_{m})\pi_{W\mid m}^{\top} \widetilde{1}  - \widetilde{\alpha}_{m} \pi_{W\mid m}^{\top}    \right\rbrace  \widetilde{\mu}_{W}  \right]\pi_{Y\mid m}^{\top}\iota_{y} .
    \label{eq:varphi_pre_left_mult}
\end{align}
Now, when $\eta_{W}\to 0$
\[
    \left\lbrace \diag\left\lbrace\pi_{(Z,X),m}\right\rbrace \widetilde{\mu}_{Z}\right\rbrace^{\top} \left\lbrace \diag(1+ \eta_{W} \widetilde{\mu}_{W})\right\rbrace^{-1} \to \left\lbrace \diag\left\lbrace\pi_{(Z,X),m}\right\rbrace \widetilde{\mu}_{Z}\right\rbrace^{\top} .
\]
Then, left multiplying \eqref{eq:varphi_pre_left_mult} by $\left\lbrace \diag\left\lbrace\pi_{(Z,X),m}\right\rbrace \widetilde{\mu}_{Z}\right\rbrace^{\top}$, we see that when $\eta_{W} \to 0$, $\varphi(P)$ converges to
\begin{align}
    &\gamma   \sum_{m}   \mu_{X}(S^{X}_{m}) \tau_{m} a_{m} + \sum_{m}   \mu_{X}(S^{X}_{m}) \tau_{m} b_{m}
    \label{eq:expression1_varphi}
\end{align}
where we defined   
\begin{align*}
    &\tau_{m} = \sum_{l} \pi_{(Z,X),l,m} \mu_{Z}(S^{Z}_{l}),
    \\
    &a_{m} =        \pi_{W\mid m}^{\top} \diag(\widetilde{\alpha}_{m}) M^{m}\iota_{y}   - 
    \frac{\pi_{W\mid m}^{\top} \widetilde{\alpha}_{m}\pi_{W\mid m}^{\top}  }{\pi_{W\mid m}^{\top}\widetilde{1}}     M^{m}\iota_{y} ,
    \\
    & b_{m}=\sum_{m} \tau_{m}  \mu_{X}(S^{X}_{m})\frac{\pi_{W\mid m}^{\top}}{\pi_{W\mid m}^{\top} \widetilde{1}} \left[ \widetilde{\alpha}_{m} + 
    \left\lbrace  \diag(\widetilde{\alpha}_{m})\pi_{W\mid m}^{\top} \widetilde{1}  - \widetilde{\alpha}_{m} \pi_{W\mid m}^{\top}    \right\rbrace  \widetilde{\mu}_{W}  \right]\pi_{Y\mid m}^{\top}\iota_{y} .
\end{align*} 
Note that $\tau_{m}$ is strictly positive.

Recall that $M_{m}$ is such that
\begin{align*}
    & M^{m} \iota_{y}= \widetilde{\alpha}_{m},
\end{align*}
and hence 
\begin{align}
    a_{m}=\left[\pi_{W\mid m}^{\top} \diag(\widetilde{\alpha}_{m}) \widetilde{\alpha}_{m}   - 
    \frac{\pi_{W\mid m}^{\top} \widetilde{\alpha}_{m}\pi_{W\mid m}^{\top}   \widetilde{\alpha}_{m}}{\pi_{W\mid m}^{\top}\widetilde{1}}\right]
    \nonumber
\end{align}
The expression in the last display is equal to 
\begin{equation}
    \left( \sum\limits_{j}\pi_{W\mid X,j,m} \right)\left\lbrace \frac{\sum\limits_{j} \pi_{W\mid X,j,m} \widetilde{\alpha}^{2}_{m,j}}{\sum\limits_{j}\pi_{W\mid X,j,m}} - \left(\frac{\sum\limits_{j} \pi_{W\mid X,j,m} \widetilde{\alpha}_{m,j}}{\sum\limits_{j}\pi_{W\mid X,j,m}}\right)^{2}\right\rbrace
\label{eq:variance_alphatilde_discrete}
\end{equation}
By \eqref{eq:lowerbound_var_alphatilde}, 
\begin{equation}
    var_{\widetilde{P}}\left\lbrace \alpha_{\widetilde{P}}(W,X)\mid X\right\rbrace > 0
\nonumber
\end{equation}
with positive probability under $\widetilde{P}$. Note that this implies that there exists $m$ such that $\widetilde{\alpha}_{m}$ is not a constant vector. Indeed, if for all $m$ it holds that $\widetilde{\alpha}_{j,m}=\beta_{m}$ for some $\beta_{m}$ then
\begin{align*}
    \alpha_{\widetilde{P}}(W,X)=\sum\limits_{j} \sum\limits_{m} \beta_{m} I\left\lbrace W \in S^{W}_{j} \right\rbrace I\left\lbrace X \in S^{X}_{m} \right\rbrace&=\sum\limits_{j} I\left\lbrace W \in S^{W}_{j} \right\rbrace \sum\limits_{m} \beta_{m} I\left\lbrace X \in S^{X}_{m} \right\rbrace  
    \\
    &=\sum\limits_{m} \beta_{m} I\left\lbrace X \in S^{X}_{m} \right\rbrace
\end{align*}
with probability one under $\widetilde{P}$, because $\sum\limits_{j} I\left\lbrace W \in S^{W}_{j} \right\rbrace =1$ with probability one under $\widetilde{P}$. In this case $\alpha_{\widetilde{P}}(W,X)$ is a function of $X$ only, and hence the variance of $\alpha_{\widetilde{P}}(W,X)$ given $X$ is zero almost surely, which is a contradiction.

Then, there exists an $m$ such that $\widetilde{\alpha}_{m}$ is not a constant vector. Since $\pi_{W\mid X,j,m}>0$ for all $j,m$, this implies that the expression in \eqref{eq:variance_alphatilde_discrete} is positive for at least one $m$, and thus $a_{m}$ is non-negative for all $m$ and positive for at least one $m$. 
Choosing
$$
\gamma = \gamma(\zeta)=\frac{\zeta - \sum_{m}\mu_{X}(S^{X}_{m}) \tau_{m} b_{m}}{\sum_{m}\mu_{X}(S^{X}_{m}) \tau_{m} a_{m}},
$$
we have that $\varphi(P)\to \zeta$ when $\eta_{W}\to 0$, which is what we wanted to show.

\subsection{Proof of Corollary \ref{coro:confidence}}

Let $P^{\ast}$ satisfy Condition \ref{cond:Past}.
Let $(\zeta_{k})_{k\geq 1}$ be a sequence contained in $\mathcal{S}$ such that $\zeta_{k}\to \sup \mathcal{S}$ and $\zeta_{k+1}\geq \zeta_{k}$ for all $k$. By Theorem \ref{theo:main} there exists  a sequence $(P_{k})_{k\geq 1}$ contained in $\mathcal{M}$ such that 
$\varphi(P_{k})=\zeta_{k}$ for all $k$ and $\Vert P_{k} - P^{\ast}\Vert_{TV}\to 0$. Fix a value of $n$ and let $P^{n}_{k}$ be the product measure of $P_{k}$ with itself $n$ times, and define $P^{\ast,n}$ analogously. Since $\Vert P_{k} - P^{\ast}\Vert_{TV}\to 0$, it follows that $\Vert P^{n}_{k} - P^{\ast,n}\Vert_{TV}\to 0$; see Chapter 3, Section 5 of \cite{pollard}.

Then,
\begin{align*}
     P^{\ast}\left\lbrace  \sup C_{n} \geq \sup \mathcal{S} \right\rbrace &\geq P^{\ast}\left\lbrace \sup C_{n} \geq \zeta_{k} \text{ for all k}\right\rbrace 
    \\
    &= \lim_{k} P^{\ast}\left\lbrace \sup C_{n} \geq \zeta_{k}  \right\rbrace
    \\
    &\geq  \liminf_{k} P^{\ast}\left\lbrace \zeta_{k} \in C_{n} \right\rbrace
    \\
    &=   \liminf_{k} \left[ P_{k}\left\lbrace \zeta_{k} \in C_{n} \right\rbrace + P^{\ast}\left\lbrace \zeta_{k} \in C_{n} \right\rbrace -  P_{k}\left\lbrace \zeta_{k} \in C_{n} \right\rbrace \right]
    \\
    &\geq \liminf_{k} \left[ P_{k}\left\lbrace \zeta_{k} \in C_{n} \right\rbrace -  \Vert P^{n}_{k} - P^{\ast,n}\Vert_{TV}\right]
    \\
    &=  \liminf_{k} P_{k}\left\lbrace \zeta_{k} \in C_{n} \right\rbrace -  \lim_{k} \Vert P^{n}_{k} - P^{\ast,n}\Vert_{TV}
    \\
    &=  \liminf_{k} P_{k}\left\lbrace \varphi(P_{k}) \in C_{n} \right\rbrace 
    \\
    &\geq \inf_{P\in\mathcal{M}}P \left\lbrace \varphi(P) \in C_{n} \right\rbrace.
\end{align*}
where in the first equality in the last display, we used the fact that the events $\lbrace \sup C_{n} \geq \zeta_{k}\rbrace $ for $k\in\mathbb{N}$ are nested. Thus 
\begin{align*}
    \liminf_{n} P^{\ast}\left\lbrace  \sup C_{n} \geq \sup \mathcal{S} \right\rbrace \geq 1-\alpha.
\end{align*}
Using an entirely analogous argument (ommited here) we can show that 
\begin{align*}
    \liminf_{n} P^{\ast}\left\lbrace  \inf C_{n} \leq \inf \mathcal{S} \right\rbrace \geq 1-\alpha ,
\end{align*}
and hence it follows
that 
\begin{align*}
    \liminf_{n} P^{\ast}\left\lbrace   \sup C_{n} \geq \sup \mathcal{S} \text{ and }\inf C_{n} \leq \inf \mathcal{S} \right\rbrace \geq 1-2\alpha ,
\end{align*}
which implies
\begin{align}
    \liminf_{n} P^{\ast}\left\lbrace   diam(C_{n})\geq diam(\mathcal{S}) \right\rbrace \geq 1-2\alpha.
    \label{eq:pstar_diam}
\end{align}

Now, fix $\varepsilon$. By \eqref{eq:pstar_diam} there exists $n_{0}$ such that for all $n\geq n_{0}$,
\begin{align}
    P^{\ast}\left\lbrace   diam(C_{n})\geq diam(\mathcal{S}) \right\rbrace \geq 1-2\alpha - \varepsilon/2.
    \nonumber
\end{align}
Fix $n\geq n_{0}$. Arguing as before, we can find $P\in\mathcal{M}$ such that $\Vert P^{n} - P^{\ast,n}\Vert_{TV}<\varepsilon/2$. Thus
\begin{align}
    \sup_{P\in\mathcal{M}}P\left\lbrace   diam(C_{n})\geq diam(\mathcal{S}) \right\rbrace \geq 1-2\alpha - \varepsilon.
    \nonumber
\end{align}
Since the last display holds for all $n\geq n_{0}$, we conclude that
\begin{align}
    \liminf_{n} \sup_{P\in\mathcal{M}}P\left\lbrace   diam(C_{n})\geq diam(\mathcal{S}) \right\rbrace \geq 1-2\alpha,
    \nonumber
\end{align}
which is what we wanted to show.

\subsection{Proof of Proposition \ref{coro:minimax}}

We will use Le Cam's method. Take $\zeta_{0},\zeta_{1}\in\mathcal{S}$. We showed in the proof of Theorem \ref{theo:main} how to construct three densities $f_{0,t},f_{1,t}$ and $\widetilde{f}$ such that 
\begin{itemize}
    \item $f_{0,t},f_{1,t}$ and $\widetilde{f}$ are rectangular non-negative simple functions, with support sets $S^{Y}_{h}$, $h=1,\dots,k_{Y}$, $S^{Z}_{l}$, $l=1,\dots,k$, $S^{W}_{j}$, $j=1,\dots,k$ and $S^{X}_{m}$, $m=1,\dots,k_{X}$.
    \item If $f^{h,l,j,m}_{0,t},f^{h,l,j,m}_{1,t}$ and $f^{h,l,j,m}$ are the values that $f_{0,t},f_{1,t}$ and $\widetilde{f}$ respectively take over the set $S^{Y}_{h}\times S^{Z}_{l} \times S^{W}_{j}\times S^{X}_{m}$ then $f^{h,l,j,m}_{0,t}\to f^{h,l,j,m}$ and $f^{h,l,j,m}_{1,t}\to f^{h,l,j,m}$ when $t\to \infty$.
    \item If $P_{0,t}$ and $P_{1,t}$ have densities $f_{0,t}$ and $f_{1,t}$ respectively then $P_{0,t} \in \mathcal{M}$, $P_{1,t} \in \mathcal{M}$ and $\varphi(P_{0,t})=\zeta_{0}$, $\varphi(P_{1,t})=\zeta_{1}$ for all $t$.
\end{itemize}

Let $P^{n}_{0,t}$ be the product measure obtained by multiplying $P_{0,t}$ with itself $n$ times, and define $P^{n}_{1,t}$ analogously.
 By equation 15.14 of \cite{wainwright}, for all $t$
\begin{equation}
    \inf\limits_{\widehat{\varphi}_{n}} \sup\limits_{P\in \mathcal{M}} E_{P} \left[ \min \left\lbrace \left\vert \widehat{\varphi}_{n} -\varphi(P) \right\vert,1 \right\rbrace \right] \geq \frac{\min(\vert \zeta_{0} - \zeta_{1} \vert,1)}{4}\left(1 -\Vert P^{n}_{0,t} - P^{n}_{1,t}\Vert_{TV} \right) ,
    \label{eq:lecam}
\end{equation}
By Pinsker's inequality 
$$
\Vert P^{n}_{0,t} - P^{n}_{1,t}\Vert_{TV} \leq \sqrt{\frac{1}{2} D(P^{n}_{0,t}\Vert P^{n}_{1,t})},
$$
where $D(\cdot \Vert \cdot)$ is the Kullback-Leibler divergence. It is well known that 
$$
D(P^{n}_{0,t}\Vert P^{n}_{1,t})= n D(P_{0,t}\Vert P_{1,t}).
$$
Now 
\begin{align*}
    D(P_{0,t}\Vert P_{1,t}) &= \int \log\left( \frac{f_{0,t}}{f_{1,t}}\right) f_{0,t} d\mu  
    \\
    &=\sum_{h,l,j,m} \log\left( \frac{f^{h,l,j,m}_{0,t}}{f^{^{h,l,j,m}}_{1,t}}\right) f^{h,l,j,m}_{0,t} \mu_{Y}(S^{Y}_{h})\mu_{Z}(S^{Z}_{l})\mu_{W}(S^{W}_{j})\mu_{X}(S^{X}_{m})
    \\
    &\overset{t\to\infty}{\to }
    \sum_{h,l,j,m} \log\left( \frac{f^{h,l,j,m}}{f^{^{h,l,j,m}}}\right) f^{h,l,j,m} \mu_{Y}(S^{Y}_{h})\mu_{Z}(S^{Z}_{l})\mu_{W}(S^{W}_{j})\mu_{X}(S^{X}_{m})
    \\
    &=0.
\end{align*}
In particular, $\Vert P^{n}_{0,t} - P^{n}_{1,t}\Vert_{TV}\to 0 $, and hence by \eqref{eq:lecam}
\begin{equation}
    \inf\limits_{\widehat{\varphi}_{n}} \sup\limits_{P\in \mathcal{M}} E_{P} \left[ \min \left\lbrace \left\vert \widehat{\varphi}_{n} -\varphi(P) \right\vert,1 \right\rbrace \right] \geq \frac{\min(\vert \zeta_{0} - \zeta_{1} \vert,1)}{4},
    \nonumber
\end{equation}
which finishes the proof.
% Taking the supremum over $\zeta_{0},\zeta_{1}\in \mathcal{S}$ on the right hand side finishes the proof.

\subsection{Proof of Proposition \ref{coro:no_score_test}}
Assume for the sake of contradiction that 
$$
\limsup_{n\to\infty} \sup_{P\in\mathcal{M}}d_{LP}(P_{{\mathcal{W}}_{n}},\Phi)=0
$$
By Lemma \ref{lemma:levy}, for all $\delta>0$, there exists $L>0$ such that 
$$
\limsup_{n} \sup_{P\in\mathcal{M}} P\left(\vert {\mathcal{W}}_{n} \lbrace \varphi(P)\rbrace\vert > L\right) <\delta/4.
$$
Thus, there exists $n_{0}$ such that for all $n\geq n_{0}$
$$
\sup_{P\in\mathcal{M}} P\left(\vert \mathbb{P}_{n}\left\lbrace \psi^{1}_{\widehat{g},\widehat{q},\varphi(P)}(O)\right\rbrace - \varphi(P) \vert> L \widehat{\sigma}_{\widehat{g}_{n},\widehat{q}_{n},\varphi(P)}/\sqrt{n}\right) <\delta/2.
$$
Since 
$\widehat{\sigma}_{\widehat{g}_{n},\widehat{q}_{n},\varphi(P)} $ converges uniformly in probability to $\sigma_{P}$ and
$\sup_{P\in\mathcal{M}} \sigma_{P} <\infty$, 
there exists $n_{1}$ such that $n_{1}\geq n_{0}$ and for all $n\geq n_{1}$
$$
 \sup_{P\in\mathcal{M}} P\left( \widehat{\sigma}_{\widehat{g}_{n},\widehat{q}_{n},\varphi(P)}> 2\sup_{P\in\mathcal{M}} \sigma_{P}  \right) <\delta/2.
$$
Hence, for all $n\geq n_{1}$,
\begin{align*}
    \sup_{P\in\mathcal{M}} P\left(\vert \mathbb{P}_{n}\left\lbrace \psi^{1}_{\widehat{g},\widehat{q},\varphi(P)}(O)\right\rbrace - \varphi(P) \vert>  (2L\sup_{P\in\mathcal{M}} \sigma_{P} )/\sqrt{n}\right)<\delta.
\end{align*}
This implies that $\mathbb{P}_{n}\left\lbrace \psi^{1}_{\widehat{g},\widehat{q},\theta}(O)\right\rbrace$ is a uniformly consistent estimator of $\varphi(P)$, which contradicts Proposition \ref{coro:minimax}.

\subsection{Auxiliary results}\label{sec:aux}

\begin{lemma}\label{lemma:approximating_density}
    Assume Condition \ref{cond:product_measure} holds. Fix $P\in \mathcal{P}(\mu)$.Then for each $i\in \mathbb{N}$ there exist 
    \begin{enumerate}
        \item four collections of pairwise disjoint Borel sets 
        \begin{align*}
            &  \lbrace S^{Y,i}_{h}: h\in \lbrace 1,\dots,k^{i}_{Y}  \rbrace\rbrace \lbrace S^{Z,i}_{l}: l\in \lbrace 1,\dots,k^{i}  \rbrace\rbrace, \lbrace S^{W,i}_{j}: j\in \lbrace 1,\dots, k^{i} \rbrace \rbrace, \lbrace S^{X,i}_{m}: m\in \lbrace 1,\dots, k^{i}_{X} \rbrace \rbrace
        \end{align*} 
        such that $0<\mu_{Y}(S^{Y,i}_{h})<\infty$ for all $h\in \lbrace 1,\dots, k^{i}_{Y} \rbrace$, $0<\mu_{Z}(S^{Z,i}_{l})<\infty$ for all $l\in \lbrace 1,\dots, k^{i} \rbrace$,
        $0<\mu_{W}(S^{W,i}_{j})<\infty$ for all $j\in \lbrace 1,\dots, k^{i} \rbrace$,
        $0<\mu_{X}(S^{X,i}_{m})<\infty$ for all $m\in \lbrace 1,\dots, k^{i}_{X} \rbrace$,$k^{i}\geq 2$, $k_{Y}^{i}\geq 2$   
        and
        \item a set of positive numbers 
        \begin{align*}
            % &\lbrace \pi^{i}_{Y,h}: h\in \lbrace 1,\dots,k^{i}_{Y}  \rbrace\rbrace,
            \lbrace \pi^{i}_{h,l,j,m}: h\in \lbrace 1,\dots,k^{i}_{Y}  \rbrace, j\in \lbrace 1,\dots,k^{i}\rbrace, m\in \lbrace 1,\dots,k^{i}_{X}  \rbrace, l\in \lbrace 1,\dots, k^{i} \rbrace \rbrace,
        \end{align*} 
        \end{enumerate}
        such that the function
        \begin{align*}
            % &\widetilde{f}^{i}_{Y}(y)=\sum\limits_{h} \pi^{i}_{Y,h} I\lbrace y\in S^{Y,i}_{h}\rbrace, 
            % \\
            &\widetilde{f}^{i}(y,z,z,x)=\sum\limits_{h,l,j,m} \pi^{i}_{h,j,l,m} I\lbrace y\in S^{Y,i}_{h}\rbrace I\lbrace w\in S^{W,i}_{j}\rbrace I\lbrace z\in S^{Z,i}_{l}\rbrace I\lbrace x\in S^{X,i}_{m}\rbrace,
        \end{align*}
        satisfies that it is non-negative, integrates to 1 under $\mu$, and 
        \begin{align*}
            % &\widetilde{f}^{i}_{Y} \overset{\text{a.e. }\mu}{\to} f_{P^\ast, Y},
            % \\
            &\widetilde{f}^{i} \overset{\text{a.e. }\mu}{\to} f_{P}, \text{ when }i\to\infty.
        \end{align*} 
      Moreover, we can: (i) choose the sets $S^{Y,i}_{h}$ such that the vector with coordinates $\int_{S^{Y,i}_{h}} y \: d\mu_{Y}$, $h=1,\dots,k^{i}_{Y}$ is not collinear with the vector of coordinates $\mu_{Y}(S^{Y,i}_{h})$, $h=1,\dots,k^{i}_{Y}$, and (ii) choose the support sets
        such that if any of the variables $Y,Z,W$ or $X$ has a finite range, then its support sets are singletons, each containing an element from the variable's range, and their union forms the variable's support.

\end{lemma}
\begin{proof}[Proof of Lemma \ref{lemma:approximating_density}]
 Take $P\in\mathcal{P}(\mu)$. We will prove the main part of the lemma first and then handle the final statements regarding the sets $S^{Y,i}_{h}$ and  variables with finite range.
    
    We will split the proof according to whether item 1 or 2 of Condition \ref{cond:product_measure} holds.

    Assume first that part 1 holds. Using arguments similar to those of Lemma A.2 of \cite{Canay2013} we can show that there exist
    \begin{enumerate}
        \item six collections of pairwise disjoint Borel sets 
        \begin{align*}
            &  \lbrace S^{Y,i}_{h}: h\in \lbrace 1,\dots,k^{i}_{Y}  \rbrace\rbrace \lbrace S^{Z_{1},i}_{l_{1}}: l_{1}\in \lbrace 1,\dots,k_{Z_{1}}^{i}  \rbrace\rbrace,
            \lbrace S^{Z_{2},i}_{l_{2}}: l_{2}\in \lbrace 1,\dots,k_{Z_{2}}^{i}  \rbrace\rbrace 
            \\
            &\lbrace S^{W_{1},i}_{j_{1}}: j_{1}\in \lbrace 1,\dots, k_{W_{1}}^{i} \rbrace \rbrace,
            \lbrace S^{W_{2},i}_{j_{2}}: j_{2}\in \lbrace 1,\dots, k_{W_{2}}^{i} \rbrace \rbrace, \lbrace S^{X,i}_{m}: m\in \lbrace 1,\dots, k^{i}_{X} \rbrace \rbrace
        \end{align*} 
        such that $0<\mu_{Y}(S^{Y,i}_{h})<\infty$ for all $h\in \lbrace 1,\dots, k^{i}_{Y} \rbrace$, $0<\mu_{Z_{1}}(S^{Z_{1},i}_{l_{1}})<\infty$ for all $l_{1}\in \lbrace 1,\dots, k_{Z_{1}}^{i} \rbrace$,$0<\mu_{Z_{2}}(S^{Z_{2},i}_{l_{2}})<\infty$ for all $l_{2}\in \lbrace 1,\dots, k_{Z_{2}}^{i} \rbrace$,
        $0<\mu_{W_{1}}(S^{W_{1},i}_{j_{1}})<\infty$ for all $j_{1}\in \lbrace 1,\dots, k_{W_{1}}^{i} \rbrace$,$0<\mu_{W_{2}}(S^{W_{2},i}_{j_{2}})<\infty$ for all $j_{2}\in \lbrace 1,\dots, k_{W_{2}}^{i} \rbrace$
        $0<\mu_{X}(S^{X,i}_{m})<\infty$ for all $m\in \lbrace 1,\dots, k^{i}_{X} \rbrace$   
        and
        \item positive numbers $\pi^{i}_{h,l_{1},l_{2},j_{1},j_{2},m}$
        \end{enumerate}
        such that the function
        \begin{align*}
            &\widetilde{f}^{i}(y,z_{1},z_{2},w_{1},w_{2},x)=
            \\
            &\sum\limits_{h,l_{1},l_{2},j_{1},j_{2},m} \pi^{i}_{h,l_{1},l_{2},j_{1},j_{2},m} I\lbrace y\in S^{Y,i}_{h}\rbrace I\lbrace w_{1}\in S^{W_{1},i}_{j_{1}}\rbrace I\lbrace w_{2}\in S^{W_{2},i}_{j_{2}}\rbrace I\lbrace z_{1}\in S^{Z_{1},i}_{l_{1}}\rbrace I\lbrace z_{2}\in S^{Z_{2},i}_{l_{2}}\rbrace I\lbrace x\in S^{X,i}_{m}\rbrace,
        \end{align*}
        satisfies that it is non-negative, integrates to 1 under $\mu$, and 
        \begin{align*}
            &\widetilde{f}^{i} \convlp f_{P} \text{ when } i\to\infty.
        \end{align*}
        Since convergence in $L^{1}(\mu)$ implies convergence in measure under $\mu$, and every sequence converging in measure has a subsequence that converges almost everywhere, passing to a subsequence if needed we can assume that 
        \begin{align*}
            % &\widetilde{f}^{i}_{Y} \overset{\text{a.e. }\mu}{\to} f_{P^\ast, Y},
            % \\
            &\widetilde{f}^{i} \overset{\text{a.e. }\mu}{\to} f_{P}.
        \end{align*} 
        Lemma A.2 of \cite{Canay2013} provides an approximating density like the former, but only for dominating measures that are a product of three $\sigma$-finite Borel measures. It is straightforward to adapt the proof to our case. We ommit the details. 
    
    The sequence of functions $(\widetilde{f}^{i})_{i\geq 1}$ satisfies all the requirements in Lemma \ref{lemma:approximating_density}, except possibly for the fact that Lemma \ref{lemma:approximating_density} requires that $k^{i}_{Z}$, the number of terms in the summation corresponding to $(l_1,l_2)$  is equal to $k^{i}_{W}$, the number of terms in the summation corresponding to $(j_1,j_2)$. Note that $k^{i}_{Z}=k^{i}_{Z_{1}}\times k^{i}_{Z_{2}}$ and $k^{i}_{W}=k^{i}_{W_{1}}\times k^{i}_{W_{2}}$. Next, we will show that we can always assume that $k^{i}_{Z_{1}}\times k^{i}_{Z_{2}} =  k^{i}_{W_{1}}\times k^{i}_{W_{2}}$. To do so, it suffices to show that we can always increase both $k^{i}_{Z_{1}}$ and $k^{i}_{W_{1}}$ by one, by splitting the sets $S^{Z_{1},i}_{l_{1}}$ and $S^{W_{1},i}_{j_{1}}$ into smaller pieces that still have positive $\mu$ measure.

    We will show that we can increase $k^{i}_{Z_{1}}$ by one, the proof for $k^{i}_{W_{1}}$  is entirely analogous. Since $\mu_{Z_{1}}$ is atomless by assumption, by Corollary 1.12.10 of \cite{Bogachev}, there exist disjoint Borel sets $C^{i,1}
    $ and $C^{i,2}$ such that 
    \begin{itemize}
        \item $\mu_{Z_{1}}\left(C^{i,t}\right) >0$ for $t=1,2$,
        \item $S^{Z_{1},i}_{k^{i}_{Z_{1}}}=C^{i,1} \cup C^{i,2}$, except for a set of $\mu$ measure zero.
    \end{itemize}
    Define $S_{k^{i}_{Z_{1}}}^{Z_{1},i,\ast}=C^{i,1}$ and $S_{k^{i}_{Z_{1}}+1}^{Z_{1},i,\ast}=C^{i,2}$. For $l_{1}\in\lbrace 1,\dots, k^{i}_{Z_{1}}-1\rbrace$, let $S_{l_{1}}^{Z_{1},i,\ast}=S^{Z_{1},i}_{l_{1}}$. 
    Also define $\pi_{h,k^{i}_{Z_{1}}+1,l_{2},j_{1},j_{2},m}= \pi_{h,k^{i}_{Z_{1}},l_{2},j_{1},j_{2},m}$  
    for all $h,l_{2},j_{1},j_{2},m$.
    Then $\widetilde{f}^{i} $ is equal almost everywhere under $\mu$ to
    \begin{equation}
        \sum\limits_{l_{1}=1}^{k^{i}_{Z_{1}}+1} \sum\limits_{h,l_{2},j_{1},j_{2},m} \pi_{h,l_{1},l_{2},j_{1},j_{2},m} I\lbrace y\in S^{Y,i}_{h}\rbrace I\lbrace w_{1}\in S^{W_{1},i}_{j_{1}}\rbrace I\lbrace w_{2}\in S^{W_{2},i}_{j_{2}}\rbrace I\lbrace z_{1}\in S^{Z_{1},i,*}_{l_{1}}\rbrace I\lbrace z_{2}\in S^{Z_{2},i}_{l_{2}}\rbrace I\lbrace x\in S^{X,i}_{m}\rbrace,
        \nonumber
    \end{equation}
    which has a number of terms in the summations corresponding to $Z_{1}$ equal to $k^{i}_{Z_{1}}+1$. 
    This finishes the proof of the main part of the lemma in the case in which part 1 of condition \ref{cond:product_measure} holds.

    Now assume that part 2 of condition \ref{cond:product_measure} holds. Arguing as before we can build a sequence of functions of the form 
    \begin{align*}
        &\widetilde{f}^{i}(y,z,w,x)=
    \sum\limits_{h,l,j,m} \pi^{i}_{h,l,j,m} I\lbrace y\in S^{Y,i}_{h}\rbrace I\lbrace w\in S^{W,i}_{j}\rbrace I\lbrace z\in S^{Z,i}_{l}\rbrace I\lbrace x\in S^{X,i}_{m}\rbrace,
    \end{align*}
    which satisfies that it is non-negative, integrates to 1 under $\mu$, and 
    \begin{align*}
        &\widetilde{f}^{i} \convlp f_{P}.
    \end{align*} 
    Since $Z$ and $W$ are discrete and their supports have equal cardinality, we can always express $\widetilde{f}^{i}$ in the form
    \begin{align*}
        &\widetilde{f}^{i}(y,z,w,x)=
    \sum\limits_{h,m} \sum\limits_{l=1}^{k^{i}_{Z}} \sum\limits_{j=1}^{k^{i}_{W}}\pi^{i}_{h,l,j,m} I\lbrace y\in S^{Y,i}_{h}\rbrace I\lbrace w=w_{j}\rbrace I\lbrace z=z_{l}\rbrace I\lbrace x\in S^{X,i}_{m}\rbrace,
    \end{align*}
    where $\lbrace w_{1},\dots,w_{k^{i}_{W}}\rbrace$ is a subset of the support of $W$ and likewise $\lbrace z_{1},\dots,z_{k^{i}_{Z}}\rbrace$ is a subset of the support of $Z$.

     If $k^{i}_{W}=k^{i}_{Z}$ then $\widetilde{f}^{i}$ already satisfies the requirements in Lemma \ref{lemma:approximating_density}. Otherwise, suppose that $k^{i}_{W}<k^{i}_{Z}$. The case $k^{i}_{W}>k^{i}_{Z}$ is handled similarly. 
    Since by assumption the supports of $Z$ and $W$ have the same cardinality, there exist $w_{j}$ for $j=k^{i}_{W} +1,\dots, k^{i}_{Z}$ that are in the support of $W$. 
    Let $\tau>0$. For any $h,l,m$ and $j=k^{i}_{W} +1,\dots, k^{i}_{Z}$ let
    $\widehat{\pi}^{i}_{h,l,j,m}=\tau$ and for $j=1,\dots, k^{i}_{W}$ let  $\widehat{\pi}^{i}_{h,l,j,m}={\pi}^{i}_{h,l,j,m}$. Let
    \[
        J = \sum\limits_{h,m} \sum\limits_{l=1}^{k^{i}_{Z}} \sum\limits_{j=1}^{k^{i}_{Z}} \widehat{\pi}_{h,l,j,m} \mu_{Y}(S^{i}_{Y,h}) \mu_{X}(S^{i}_{X,m}) .     
    \]
    Define 
    \begin{align*}
        \widehat{f}^{i}(y,z,w,x) = 
        (1/J) \sum\limits_{h,m} \sum\limits_{l=1}^{k^{i}_{Z}} \sum\limits_{j=1}^{k^{i}_{Z}} \widehat{\pi}^{i}_{h,l,j,m} I\lbrace y\in S^{Y,i}_{h}\rbrace I\lbrace w=w_{j}\rbrace I\lbrace z=z_{l}\rbrace I\lbrace x\in S^{X,i}_{m}\rbrace.
    \end{align*}
    Note that $\widehat{f}^{i}$ satisfies that $\widehat{f}^{i}\geq 0$, $\int \widehat{f}^{i} d\mu =1$, and that $\widehat{\pi}^{i}_{h,l,j,m}/J>0$ for all $h,l,j,m$. Moreover, there are an equal number of terms corresponding to indices $j$ and to $l$ in the summation that defines $\widehat{f}^{i}(y,z,w,x)$.
    For each $i$ take $\tau$ small enough so that $\Vert \widehat{f}^{i} - \widetilde{f}^{i}\Vert_{L^{1}(\mu)}< 1/i$. Then $\widehat{f}^{i} \convlp f_{P}$ as $i\to\infty$. Passing to a subsequence, we have that $\widehat{f}^{i} \convas f_{P}$ as $i\to\infty$. 
    This finishes the proof of the main part lemma for the case in which part 2 of condition \ref{cond:product_measure} holds.

    Next, we show that we can choose the sets $S^{Y,i}_{h}$ such that the vector with coordinates $\int_{S^{Y,i}_{h}} y \: d\mu_{Y}$, $h=1,\dots,k^{i}_{Y}$ is not collinear with the vector of coordinates $\mu_{Y}(S^{Y,i}_{h})$, $h=1,\dots,k^{i}_{Y}$ and that we can assume that $k^{i}_{Y}\geq 2$. If $\mu_{Y}$ is the counting measure on a finite set with two or more elements, the statement holds trivially. Assume then that $\mu_{Y}$ is an atomless measure. Suppose that there exists $c\in\mathbb{R}$ such that for all $h$ it holds that $\int_{S^{Y,i}_{h}} y \: d\mu_{Y}= c \times \mu_{Y}(S^{Y,i}_{h})$. Take any $h$ and define 
    \begin{align*}
    &S_{h}^{Y,i,+} = S^{Y,i}_{h} \cap \lbrace y > c\rbrace
    \\
    &S_{h}^{Y,i,-} = S^{Y,i}_{h} \cap \lbrace y \leq c\rbrace.
    \end{align*}
    Since $\mu_{Y}$ is atomless, it does not  have an atom at $c$. Then, $S^{Y,i}_{h} = S_{h}^{Y,i,+} \cup S_{h}^{Y,i,-}$, and both $S_{h}^{Y,i,+}$ and $S_{h}^{Y,i,-}$ have positive measure, because $\int_{S^{Y,i}_{h}} y \: d\mu_{Y}= c \times \mu_{Y}(S^{Y,i}_{h})$. Moreover,  
    $$
        \frac{1}{\mu_{Y}(S_{h}^{Y,i,+})}\int_{S^{Y,i,+}_{h}} y \: d\mu_{Y} >c 
    $$
    and
    $$
        \frac{1}{\mu_{Y}(S_{h}^{Y,i,-})}\int_{S^{Y,i,-}_{h}} y \: d\mu_{Y} <c. 
    $$
    Thus, by splitting the set $S_{h}^{Y,i}$ like we just described, we see that we can assume that the vector with coordinates $\int_{S^{Y,i}_{h}} y \: d\mu_{Y}$, $h=1,\dots,k^{i}_{Y}$ is not collinear with the vector of coordinates $\mu_{Y}(S^{Y,i}_{h})$, $h=1,\dots,k^{i}_{Y}$. To prove that we can assume that $k^{i}_{Y}\geq 2$ we proceed similarly, splitting the sets 
    $S^{Y,i}_{h}$ into smaller pieces if needed.

        To prove the last statement of the lemma, regarding variables with finite range, we can use arguments similar to the ones we used to show that we can assume that $k^{i}_{Z}=k^{i}_{W}$ when $Z$ and $W$ are discrete. We ommit the details.

\end{proof}

\begin{lemma}\label{lemma:positive_cond_var}
    Assume Condition \ref{cond:alpha_convergence_1} holds. Let $P\in\mathcal{P}(\mu)$ be a law satisfying \eqref{eq:cont}. Assume that $(P_{i})_{i\geq 1}$ is a sequence of laws in $\mathcal{P}(\mu)$ that satisfy \eqref{eq:cont} and such that $f_{P_{i}}\convas f_{P}$ as $i\to\infty$. Assume that $var_{P}\left\lbrace \alpha_{P}(W,X)\mid X\right\rbrace>0$ with positive probability under $P$. 
    Then, there exists $i_{0}$ such that for all $i\geq i_{0}$, $var_{P_{i}}\left\lbrace \alpha_{P_{i}}(W,X)\mid X\right\rbrace>0$ with positive probability under $P_{i}$.
\end{lemma}
\begin{proof}
    Assume, for the sake of contradiction, that there exists a subsequence $(P_{i_{q}})_{q\geq 1}$ such that 
    $$
    var_{P_{i_{q}}}\left\lbrace \alpha_{P_{i_{q}}}(W,X)\mid X\right\rbrace=0
    $$ 
    almost surely under $P_{i_{q}}$ for all $q$. Since 
    by assumption $f_{P_{i}}\convas f_{P}$ as $i\to\infty$, by Scheffe's lemma we have that $P_{i_{q}}$ converges to $P$ in total variation, that is
    $$
    \Vert P_{i_{q}}-P\Vert_{TV} \to 0.
    $$
    Eventually passing to a subsequence if needed, we can assume that 
    $$
    \Vert P_{i_{q}}-P\Vert_{TV} < 1/q^{2},
    $$
    for all $q$.
    This implies that 
    $$
    P\left[
    var_{P_{i_{q}}}\left\lbrace \alpha_{P_{i_{q}}}(W,X)\mid X\right\rbrace=0
    \right] \geq P_{i_{q}}\left[
    var_{P_{i_{q}}}\left\lbrace \alpha_{P_{i_{q}}}(W,X)\mid X\right\rbrace=0
    \right] - 1/q^{2} = 1 - 1/q^{2}.
    $$
    Thus
    $$
    P\left[
    var_{P_{i_{q}}}\left\lbrace \alpha_{P_{i_{q}}}(W,X)\mid X\right\rbrace\neq 0
    \right] \leq 1/q^{2},
    $$
    for all $q$.
    Let $K_{1}$ be the event in which there exists $q_{0}$ such that for all $q\geq q_{0}$, $var_{P_{i_{q}}}\left\lbrace \alpha_{P_{i_{q}}}(W,X)\mid X\right\rbrace=0$.
    By the Borel-Cantelli lemma, $P(K_{1})=1$.

    Now, by Condition \ref{cond:alpha_convergence_1}, eventually passing to a subsequence if needed, we can assume that $\alpha_{P_{i_{q}}}(W,X)$ converges to $\alpha_{P}(W,X)$ almost surely under $P$. Then, letting $K_{2}$ be the event in which $\alpha_{P_{i_{q}}}(W,X)\to \alpha_{P}(W,X)$, we have that $P(K_{2})=1$.
    
    In the event $K_{1}\cap K_{2}$, we have that, for all sufficiently large $q$, $var_{P_{i_{q}}}\left\lbrace \alpha_{P_{i_{q}}}(W,X)\mid X\right\rbrace=0$. This means that $\alpha_{P_{i_{q}}}(W,X)$ is equal to a function solely of $X$. But in the event $K_{1}\cap K_{2}$ we also have that $\alpha_{P_{i_{q}}}(W,X)\to \alpha_{P}(W,X)$. Thus, it has to be that $\alpha_{P}(W,X)$ is equal to a function solely of $X$. Since $P(K_{1}\cap K_{2})=1$, we have that $var_{P}\left\lbrace \alpha_{P}(W,X)\mid X\right\rbrace=0$ almost surely under $P$. This contradicts the assumption that $var_{P}\left\lbrace \alpha_{P}(W,X)\mid X\right\rbrace>0$ with positive probability under $P$.
\end{proof}

\begin{lemma}\label{lemma:conv_alpha_params}
   Assume Condition \ref{cond:alpha_convergence_1} holds.
    Let $P$ and $(P_{i})_{i\geq 1}$ be laws in $\mathcal{P}(\mu)$ that satisfy \eqref{eq:cont} and that have densities that are rectangular non-negative simple functions. 
    Assume that $f_{P_{i}}\convas f_{P}$ as $i\to\infty$,
    that 
    $\alpha_{P}$ coincides $P$-almost surely with a function of the form
    \begin{equation}
        (w,x)\mapsto \sum\limits_{j=1}^{k_{W}}\sum\limits_{m=1}^{k_{X}} \alpha_{j,m}  I\lbrace (w,x)\in (S^{W}_{j}\times S^{X}_{m})\rbrace,
        \nonumber
    \end{equation}
    and that $\alpha_{P_{i}}$ coincides $P_{i}$-almost surely with a function of the form
    \begin{equation}
        (w,x)\mapsto \sum\limits_{j=1}^{k_{W}}\sum\limits_{m=1}^{k_{X}} \alpha^{i}_{j,m}  I\lbrace (w,x)\in (S^{W}_{j}\times S^{X}_{m})\rbrace.
        \nonumber
    \end{equation}
    Then $\alpha^{i}_{j,m} \to \alpha_{j,m}$ for all $j,m$ when $i\to\infty$.
\end{lemma}
\begin{proof}
    Let $E_{i}$ be the event in which
    $\alpha_{P_{i}}$ coincides with a function of the form
    \begin{equation}
        (w,x)\mapsto \sum\limits_{j=1}^{k_{W}}\sum\limits_{m=1}^{k_{X}} \alpha^{i}_{j,m}  I\lbrace (w,x)\in (S^{W}_{j}\times S^{X}_{m})\rbrace.
        \nonumber
    \end{equation}
    By assumption $P_{i}(E_{i})=1$ for all $i$.

    Since 
    we are assuming that $f_{P_{i}}\convas f_{P}$ as $i\to\infty$, by Scheffe's lemma we have that $P_{i}$ converges to $P$ in total variation, that is
    $$
    \Vert P_{i}-P\Vert_{TV} \to 0.
    $$
    Eventually passing to a subsequence if needed, we can assume that 
    $$
    \Vert P_{i}-P\Vert_{TV} < 1/i^{2},
    $$
    for all $i$.
    This implies that 
    $$
    P\left[ E_{i}
    \right] \geq P_{i}\left[
    E_{i}
    \right] - 1/i^{2} = 1 - 1/i^{2}.
    $$
    Thus
    $$
    P\left[
    E^{c}_{i}
    \right] < 1/i^{2},
    $$
    for all $i$.
    Let $K_{1}$ be the event in which there exists $i_{0}$ such that for all $i\geq i_{0}$, $E_{i}$ holds.
    By the Borel-Cantelli lemma, $P(K_{1})=1$.
    
    Let $K_{2}$ be the event in which $\alpha_{P}$ coincides with a function of the form
    \begin{equation}
        (w,x)\mapsto \sum\limits_{j=1}^{k_{W}}\sum\limits_{m=1}^{k_{X}} \alpha_{j,m}  I\lbrace (w,x)\in (S^{W}_{j}\times S^{X}_{m})\rbrace.
        \nonumber
    \end{equation}
    By assumption $P(K_{2})=1$.
    
    Note that
    \begin{align*}
        &P\left\lbrace \alpha_{P_{i}}(W,X) \to \alpha_{P}(W,X) \right\rbrace =
        \\
        &\sum\limits_{j=1}^{k_{W}}\sum\limits_{m=1}^{k_{X}} P\left[ \left\lbrace \alpha_{P_{i}}(W,X) \to \alpha_{P}(W,X)\right\rbrace \cap (K_{1}\cap K_{2}) \mid (W,X) \in  (S^{W}_{j}\times S^{X}_{m})\right] P\left\lbrace (W,X) \in  (S^{W}_{j}\times S^{X}_{m})\right\rbrace=
        \\
        &\sum\limits_{j=1}^{k_{W}}\sum\limits_{m=1}^{k_{X}} I\left\lbrace \alpha^{i}_{j,m} \to \alpha_{j,m} \right\rbrace P\left\lbrace (W,X) \in  (S^{W}_{j}\times S^{X}_{m})\right\rbrace,
    \end{align*}
    where we used the fact that $P(K_{1}\cap K_{2})=1$, and that in the event $K_{1}\cap K_{2}\cap \lbrace (W,X)\in S^{W}_{j}\times S^{X}_{m}\rbrace$ it holds that $\alpha_{P}(W,X)=\alpha_{j,m}$ and $\alpha_{P_{i}}(W,X)=\alpha^{i}_{j,m}$ for all sufficiently large $i$.
    Since $P\left\lbrace \alpha_{P_{i}}(W,X) \to \alpha_{P}(W,X) \right\rbrace=1$ by assumption,   $P\left\lbrace (W,X) \in  (S^{W}_{j}\times S^{X}_{m})\right\rbrace$ is positive for all $j,m$ and 
    $$
    \sum\limits_{j=1}^{k_{W}}\sum\limits_{m=1}^{k_{X}}P\left\lbrace (W,X) \in  (S^{W}_{j}\times S^{X}_{m})\right\rbrace=1
    $$ 
    we have that $I\left\lbrace \alpha^{i}_{j,m} \to \alpha_{j,m} \right\rbrace=1$ for all $j,m$. This implies that $\alpha^{i}_{j,m} \to \alpha_{j,m}$ for all $j,m$ when $i\to\infty$, which is what we wanted to show. 
\end{proof}

\begin{lemma}\label{lemma:intermediate_value}
    Let $\psi(\gamma,\eta)$ be a continuous function defined over $\mathbb{R}\times\left(\mathbb{R} \setminus \lbrace 0 \rbrace \right) $ that satisfies:
    \begin{enumerate}
        % \item For all $\gamma\in \mathbb{R}$, $\psi(\gamma,\eta)$ converges when $\eta\to 0$
        \item For all $\zeta\in \mathbb{R}$ there exists $\gamma(\zeta)$ such that $\psi(\gamma(\zeta),\eta)\to \zeta$ when $\eta\to 0$.
        \item $\gamma(\zeta)$ is increasing in $\zeta$.
    \end{enumerate}
    Then for all $\zeta \in \mathbb{R}$ there exists a sequence $\lbrace(\gamma_{t},\eta_{t})\rbrace_{t\geq 1}$ such that $(\gamma_{t})_{t\geq 1}$ is bounded, $\eta_{t} \to 0$ when $t\to \infty$ and $\psi(\gamma_{t},\eta_{t})=\zeta$ for all $t$.
\end{lemma}
\begin{proof}[Proof of Lemma \ref{lemma:intermediate_value}]
    Take $\zeta\in \mathbb{R}$. For each $\eta \neq 0$ define the function $\Psi_{\eta}$ by
    $$
    \Psi_{\eta}(\gamma^{\prime})= \psi(\gamma^{\prime},\eta) - \zeta.
    $$
    Then $\Psi_{\eta}$ is continuous over $\mathbb{R}$ for every $\eta\neq 0$. 
    Fix $\delta>0$ and take any sequence $(\eta_{t})_{t\geq 1}$ converging to zero as $t\to\infty$. When $t\to\infty$,
    \begin{align*}
        &\Psi_{\eta_{t}}(\gamma(\zeta - \delta))=\psi\lbrace\gamma(\zeta - \delta), \eta_{t} \rbrace -\zeta  \to  -\delta,
        \\
        &\Psi_{\eta_{t}}(\gamma(\zeta + \delta))=\psi\lbrace\gamma(\zeta + \delta), \eta_{t} \rbrace -\zeta \to \delta.
    \end{align*}
    Thus, by the intermediate value theorem, for all sufficiently large $t$ there exists $\gamma_{t}\in(\gamma(\zeta- \delta) , \gamma(\zeta+ \delta) )$ such that $\Psi_{\eta_{t}}(\gamma_{t})=0$, that is, $\psi(\gamma_{t},\eta_{t}) = \zeta$, which is what we wanted to show.
\end{proof}

\begin{lemma}\label{lemma:levy}
    For $n\in\mathbb{N}$, let $B_{n,P}:=b_{n,P}(O_{1},\dots, O_{n})$, where $b_{n,P}$ is a measurable function that could depend on $P$, the law of $O$. Let $Q$ be a fixed Borel measure over $\mathbb{R}$. 
    For each $n$, let $P_{n}$ be the law of $B_{n,P}$ when $O$ has law $P$.
    Assume that
    $$  
        \limsup_{n}\sup_{P\in\mathcal{M}} d_{LP}(P_{n},Q) =0.
    $$
    Then, for every $\delta>0$ there exists $L>0$ such that
    $$
        \limsup_{n} \sup_{P\in\mathcal{M}} P\left(\vert B_{n,P}\vert > L \right) \leq \delta.
    $$
\end{lemma}
\begin{proof}[Proof of Lemma \ref{lemma:levy}]
    Take $L>0$ such that $Q\left\lbrace \mathbb{R}\setminus (-L,L)\right\rbrace < \delta/{2}$. Let $n_{0}$ be such that for all $n\geq n_{0}$ it holds that 
    $\sup_{P\in\mathcal{M}} d_{LP}(P_{n},Q) <\delta/2$. Let $\varepsilon_{1}>0$ be such that $\varepsilon_{1}<\delta/{2}$, $\varepsilon_{1}<L$ and for all $n\geq n_{0}$ and $P\in\mathcal{M}$
    $$
        P_{n} \left[ \left\lbrace \mathbb{R}\setminus (-L,L)\right\rbrace^{\varepsilon_{1}}\right] \leq Q\left\lbrace \mathbb{R}\setminus (-L,L)\right\rbrace  + \varepsilon_{1}.
    $$
    Since 
    $$
    P_{n} \left[ \left\lbrace \mathbb{R}\setminus (-L,L)\right\rbrace^{\varepsilon_{1}}\right]= P_{n} \left\lbrace \mathbb{R}\setminus [-L+\varepsilon_{1},L-\varepsilon_{1}]\right\rbrace=P\left(\vert B_{n,P}\vert > L-\varepsilon_{1}\right)
    $$
    we conclude that for all $n\geq n_{0}$  and $P\in\mathcal{M}$
    $$
    P\left(\vert B_{n,P}\vert > L-\varepsilon_{1}\right) \leq \delta/2 + \varepsilon_{1} \leq \delta.
    $$
    This finishes the proof of the lemma.
\end{proof}

\section{Validity of Conditions \ref{cond:alpha_coheres} and \ref{cond:alpha_convergence_1} for the examples}
In this section, we verify that the conditions needed to prove our results hold for the examples listed in Section \ref{sec:model}. 

\begin{proposition}\label{prop:examples}
    Conditions \ref{cond:alpha_coheres} and \ref{cond:alpha_convergence_1} hold for the examples listed in Section \ref{sec:model}.
\end{proposition}
\begin{proof}

We prove the propostion for Examples \ref{ex:ATE_proximal} and \ref{ex:ATE_IV}. The proofs for the other examples are very similar.

We begin with Example \ref{ex:ATE_proximal}. In this case $X=(A,L)$, with $A$ binary, and $\alpha_{P}(W,A,L)=(2A-1)/f_{P,A\mid(W,L)}(A\mid W,L)$. Note that for all $P$ that satisfies \eqref{eq:cont}  it holds that $f_{P,A\mid(W,L)}(a\mid W,L)>0$ for $a\in\lbrace 0, 1\rbrace$, almost surely under $P$.

We will first show that Condition \ref{cond:alpha_coheres} holds. Let $P$ be a law in $\mathcal{P}(\mu)$ with a density that is a rectangular non-negative simple function. Since $A$ is binary, we can assume that $P$ is such that
    $f_{P}$ coincides $\mu$-almost everywhere with a function of the form
    \begin{align}
        (y,z,w,a,l)\mapsto \sum\limits_{h=1}^{k^{i}_{Y}}\sum\limits_{g=1}^{k^{i}_{Z}}
        \sum\limits_{j=1}^{k^{i}_{W}} \sum\limits_{t=0}^{1} \sum\limits_{m=1}^{k^{i}_{L}} \pi^{i}_{h,g,j,t,m} I\lbrace y\in S^{Y,i}_{h} \rbrace  I\lbrace z\in S^{Z,i}_{g} \rbrace I\lbrace w\in S^{W,i}_{j} \rbrace I\lbrace t=a \rbrace I\lbrace l \in S^{L,i}_{m} \rbrace.
        \label{eq:simple_proximal}
    \end{align}
Then, if $(W,L)\in  S^{W}_{j} \times S^{L}_{m}$ for some $j,m$, which happens with probability one under $P$,  for $a\in\lbrace 0, 1\rbrace$ 
\begin{align*}
    f_{P,A\mid W,L}(a\mid W,L) = \frac{f_{P,(W,A,L)}(W,a,L)}{f_{P,(W,L)}(W,L)}&=
    \frac{\int f_{P,(Y,Z,W,A,L)}(y,z,W,a,L)d\mu_{Y}d\mu_{Z} }{\int f_{P,(Y,Z,W,A,L)}(y,z,W,a,L)d\mu_{Y}d\mu_{Z}d\mu_{A}}
    \\
    &=\frac{\sum_{h,g} \pi_{h,g,j,a,m} \mu_{Z}(S^{Z}_{g}) \mu_{Y}(S^{Y}_{h}) }{\sum_{h,g,t} \pi_{h,g,j,t,m} \mu_{Z}(S^{Z}_{g}) \mu_{Y}(S^{Y}_{h}) }>0.
\end{align*}
Hence, there exists a constant $c>0$ such that $f_{P,A\mid W,L}(a\mid W,L)>c$ for all $a\in\lbrace 0, 1\rbrace$, almost surely under $P_{i}$. This implies that $P$ satisfies \eqref{eq:cont}.
It follows that with probability one under $P$,
\begin{align}
    \alpha_{P}(W,A,L)= \sum\limits_{j,a,m} \alpha_{j,a,m} \: I\lbrace W\in S^{W}_{j}\rbrace I\lbrace A=a \rbrace I\lbrace L\in S^{L}_{m}\rbrace, 
    \label{eq:alpha_proximal}
\end{align}
where
\begin{align}
    \alpha_{j,a,m}=(2a-1)\frac{\sum_{h,g,t} \pi_{h,g,j,t,m} \mu_{Z}(S^{Z}_{g}) \mu_{Y}(S^{Y}_{h})}{\sum_{h,g} \pi_{h,g,j,a,m} \mu_{Z}(S^{Z}_{g}) \mu_{Y}(S^{Y}_{h})},
    \nonumber
\end{align}
which shows Condition \ref{cond:alpha_coheres} holds.

% Then, if $(W,L)\in  S^{W,i}_{j} \times S^{L,i}_{m}$ for some $j,m$, which happens with probability one under $P_{i}$,  for $a\in\lbrace 0, 1\rbrace$ 
% \begin{align*}
%     f_{P_{i},A\mid W,L}(a\mid W,L) = \frac{f_{P_{i},(W,A,L)}(W,a,L)}{f_{P_{i},(W,L)}(W,L)}&=
%     \frac{\int f_{P_{i},(Y,Z,W,A,L)}(y,z,W,a,L)d\mu_{Y}d\mu_{Z} }{\int f_{P_{i},(Y,Z,W,A,L)}(y,z,W,a,L)d\mu_{Y}d\mu_{Z}d\mu_{A}}
%     \\
%     &=\frac{\sum_{h,g} \pi^{i}_{h,g,j,a,m} \mu_{Z}(S^{Z,i}_{g}) \mu_{Y}(S^{Y,i}_{h}) }{\sum_{h,g,t} \pi^{i}_{h,g,j,t,m} \mu_{Z}(S^{Z,i}_{g}) \mu_{Y}(S^{Y,i}_{h}) }>0.
% \end{align*}
% Hence, there exists a constant $c>0$ such that $f_{P_{i},A\mid W,L}(a\mid W,L)>c$ for all $a\in\lbrace 0, 1\rbrace$, almost surely under $P_{i}$. This implies that $P_{i}$ satisfies \eqref{eq:cont}.
% It follows that with probability one under $P_{i}$,
% \begin{align}
%     \alpha_{P_{i}}(W,A,L)= \sum\limits_{j,a,m} \alpha^{i}_{j,a,m} \: I\lbrace W\in S^{W,i}_{j}\rbrace I\lbrace A=a \rbrace I\lbrace L\in S^{L,i}_{m}\rbrace, 
%     \label{eq:alpha_proximal}
% \end{align}
% where
% \begin{align}
%     \alpha_{j,a,m}=(2a-1)\frac{\sum_{h,g,t} \pi^{i}_{h,g,j,t,m} \mu_{Z}(S^{Z}_{g}) \mu_{Y}(S^{Y}_{h})}{\sum_{h,g} \pi^{i}_{h,g,j,a,m} \mu_{Z}(S^{Z}_{g}) \mu_{Y}(S^{Y}_{h})},
%     \nonumber
% \end{align}
% which shows Condition \ref{cond:alpha_coheres} holds. 

Next, we will verify that Condition \ref{cond:alpha_convergence_1} holds. 
Let $P$ be a law in $\mathcal{P}(\mu)$ that satisfies \eqref{eq:cont}. Let $(P_{i})_{i\geq 1}$ be a sequence of laws in $\mathcal{P}(\mu)$  such that $P_{i}$ satisfies \eqref{eq:cont} for all $i$ and $f_{P_{i}} \overset{\text{a.e. }\mu}{\to} f_{P}$. 
We need to show that there exists a subsequence $i_{q}$ such that $\alpha_{P_{i_{q}}}(W,A,L)$ converges to $\alpha_{P}(W,A,L)$ almost surely under $P$. 
Note that since $f_{P_{i}} \overset{\text{a.e. }\mu}{\to} f_{P}$, by Scheffe's lemma $f_{P_{i}} \convlp f_{P}$. Then
\begin{align*}
    &f_{P_{i},(W,A,L)}(w,a,l)=\int f_{P_{i}}(y,z,w,a,l) d\mu_{Y}d\mu_{Z} \convlp \int f_{P}(y,z,w,a,l) d\mu_{Y}d\mu_{Z} = f_{P,(W,A,L)}(w,a,l)
    \\
    &f_{P_{i},(W,L)}(w,l)=\int f_{P_{i}}(y,z,w,a,l) d\mu_{Y}d\mu_{Z}d\mu_{A} \convlp \int f_{P}(y,z,w,a,l) d\mu_{Y}d\mu_{Z}d\mu_{A}  = f_{P,(W,L)}(w,l).
\end{align*}
Thus, since convergence in $L^{1}(\mu)$ implies convergence in measure under $\mu$, and every sequence converging in measure has a subsequence converging almost surely, we can choose a subsequence $i_{q}$ so that 
\begin{align*}
    &f_{P_{i_{q}},(W,A,L)} \convas f_{P,(W,A,L)}
    \\
    &f_{P_{i_{q}},(W,L)} \convas f_{P,(W,L)}.
\end{align*}
holds. Since $P\in \mathcal{P}(\mu)$
\begin{align}
    &f_{P_{i_{q}},(W,A,L)} \to f_{P,(W,A,L)} \text{ almost surely under }P,
    \label{eq:marginal_WAL_converges}
    \\
    &f_{P_{i_{q}},(W,L)} \to f_{P,(W,L)}
    \text{ almost surely under }P
    .
    \label{eq:marginal_WL_converges}
\end{align}
Let $B_{pos}$ be the event where $f_{P,(W,L)}(W,L)>0$, $f_{P,(W,A,L)}(W,a,L)>0$ for $a\in\lbrace 0, 1\rbrace$, $f_{P_{i_{q}},(W,L)} \to f_{P,(W,L)}$ and $f_{P_{i_{q}},(W,A,L)} \to f_{P,(W,A,L)}$  hold. Then by \eqref{eq:marginal_WAL_converges} and \eqref{eq:marginal_WL_converges},
$P(B_{pos})=1$. In the event $B_{pos}$
$$
\alpha_{P_{i_{q}}}(W,A,L)=\frac{2A-1}{f_{P_{i_{q}},A\mid W,L}(A\mid W,L)}  \to \frac{2A-1}{f_{P,A\mid W,L}(A\mid W,L)},
$$
which is what we wanted to show.

Next, we turn to Example \ref{ex:ATE_IV}. In this case $\alpha_{P}(W,X)=(2W-1)/f_{P,W\mid X}(W\mid X)$. Note that for all $P$ that satisfies \eqref{eq:cont} it holds that $f_{P,W\mid X}(w\mid X)>0$ for $w\in\lbrace 0, 1\rbrace$, almost surely under $P$.

We will first show that Condition \ref{cond:alpha_coheres} holds. Let $P$ be a law in $\mathcal{P}(\mu)$ with a density that is a rectangular non-negative simple function. Since $W$ and $Z$ are binary, we can assume that $P$ is such that $f_{P}$ coincides $\mu$-almost everywhere with a function of the form
    \begin{align}
        (y,z,w,x)\mapsto \sum\limits_{h=1}^{k_{Y}}\sum\limits_{l=0}^{1}
        \sum\limits_{j=0}^{1} \sum\limits_{m=1}^{k_{X}} \pi_{h,l,j,m} I\lbrace y\in S^{Y}_{h} \rbrace  I\lbrace z=l \rbrace I\lbrace w=j \rbrace I\lbrace x \in S^{X}_{m} \rbrace.
        \label{eq:simple_ateIV}
    \end{align}
Then, if $X\in  S^{X}_{m} $, which happens with probability one under $P$,  for $w\in\lbrace 0, 1\rbrace$ 
\begin{align*}
    f_{P,W\mid X}(w\mid X) = \frac{f_{P,(W,X)}(w,X)}{f_{P,X}(X)}&=
    \frac{\int f_{P,(Y,Z,W,X)}(y,z,w,X)d\mu_{Y}d\mu_{Z} }{\int f_{P,(Y,Z,W,X)}(y,z,W,X)d\mu_{Y}d\mu_{Z}d\mu_{W}}
    \\
    &=\frac{\sum_{h,l} \pi_{h,l,w,m}  \mu_{Y}(S^{Y}_{h}) }{\sum_{h,l,j} \pi_{h,l,j,m}  \mu_{Y}(S^{Y}_{h}) }>0.
\end{align*}
Hence, there exists a constant $c>0$ such that  $f_{P,W\mid X}(w\mid X)>c$ for all $w\in\lbrace 0, 1\rbrace$, almost surely under $P$. Thus $P$ satisfies \eqref{eq:cont}.
It follows that with probability one under $P$,
\begin{align}
    \alpha_{P}(W,X)= \sum\limits_{j,m} \alpha_{j,m} \: I\lbrace w=j \rbrace I\lbrace X\in S^{X}_{m}\rbrace, 
    \label{eq:alpha_ateIV}
\end{align}
where
\begin{align}
    \alpha_{j,m}=(2w-1)\frac{\sum_{h,l,j} \pi_{h,l,j,m}  \mu_{Y}(S^{Y}_{h})}{\sum_{h,l} \pi_{h,l,w,m}  \mu_{Y}(S^{Y}_{h}) }
    \nonumber
\end{align}
which shows Condition \ref{cond:alpha_coheres} holds. 

Next, we will verify that Condition \ref{cond:alpha_convergence_1} holds. Let $P$ be a law in $\mathcal{P}(\mu)$ that satisfies \eqref{eq:cont}.
Let $(P_{i})_{i\geq 1}$ be a sequence of laws in $\mathcal{P}(\mu)$  such that $P_{i}$ satisfies \eqref{eq:cont} for all $i$ and $f_{P_{i}} \overset{\text{a.e. }\mu}{\to} f_{P}$. 
We need to show that there exists a subsequence $i_{q}$ such that $\alpha_{P_{i_{q}}}(W,X)$ converges to $\alpha_{P}(W,X)$ almost surely under $P$. 

Note that since $f_{P_{i}} \overset{\text{a.e. }\mu}{\to} f_{P}$, by Scheffe's lemma $f_{P_{i}} \convlp f_{P}$. Then,
\begin{align*}
    &f_{P_{i},(W,X)}\convlp f_{P,(W,X)},
    \\
    &f_{P_{i},X} \convlp  f_{P,X}.
\end{align*}
Thus, since convergence in $L^{1}(\mu)$ implies convergence in measure under $\mu$, and every sequence converging in measure has a subsequence converging almost surely, we can choose a subsequence $i_{q}$ so that 
\begin{align}
    &f_{P_{i_{q}},(W,X)} \convas f_{P,(W,X)}
    \nonumber
    \\
    &f_{P_{i_{q}},X} \convas f_{P,X}.
    \nonumber
\end{align}
also holds. Thus, since $P\in \mathcal{P}(\mu)$
\begin{align}
    &f_{P_{i_{q}},(W,X)} \to f_{P,(W,X)} \text{ almost surely under }P,
    \label{eq:marginal_WAL_converges_IV}
    \\
    &f_{P_{i_{q}},X} \to f_{P,X}
    \text{ almost surely under }P
    .
    \label{eq:marginal_WL_converges_IV}
\end{align}

Let $B_{pos}$ be the event where $f_{P,X}(X)>0$, $f_{P,(W,X)}(w,X)>0$ for $w\in\lbrace 0, 1\rbrace$, $f_{P_{i_{q}},(W,X)} \to f_{P,(W,X)}$ and $f_{P_{i_{q}},X} \to f_{P,X}$  hold. Then by \eqref{eq:marginal_WAL_converges_IV} and \eqref{eq:marginal_WL_converges_IV},
$P(B_{pos})=1$. In the event $B_{pos}$
$$
\alpha_{P_{i_{q}}}(W,X)=\frac{2W-1}{f_{P_{i_{q}},W\mid X}(W\mid X)}  \to \frac{2W-1}{f_{P,W\mid X}(W\mid X)},
$$
which is what we wanted to show.
\end{proof}

\bibliographystyle{apalike}
\bibliography{testability}
\end{document}